\documentclass[11pt]{article}

\usepackage[margin=1in]{geometry}
\usepackage{times}          % optional, if you want the same font feel
\usepackage{microtype}      % optional, nice for arXiv PDFs
\usepackage{graphicx}
\usepackage{subcaption}
\usepackage{booktabs} % for professional tables

\usepackage{amsmath,amssymb,amsthm,mathtools}
\usepackage{graphicx}
\usepackage[colorlinks=true,linkcolor=blue,citecolor=blue,urlcolor=blue]{hyperref}
\usepackage{algorithm}
\usepackage{algorithmic}
\usepackage{caption}

\usepackage[round,authoryear]{natbib}
\let\cite\citep

\theoremstyle{plain}
\newtheorem{theorem}{Theorem}[section]

\newtheorem{lemma}[theorem]{Lemma}

\theoremstyle{definition}

\newtheorem{assumption}[theorem]{Assumption}
\theoremstyle{remark}

\usepackage{enumitem}
\newcommand{\R}{\mathbb{R}}
\newcommand{\set}{C}
\newcommand{\supp}{S}
\newcommand{\ve}{a}
\newcommand{\Ve}{A}
\newcommand{\const}{c}
\newcommand{\Const}{M}
\newcommand{\Lip}{L}
\newcommand{\N}{\mathbb{N}}

\newcommand{\n}[1]{\| #1 \|}
\newcommand{\abs}[1]{\vert #1 \rvert}

\DeclareMathOperator{\conv}{conv}
\DeclareMathOperator{\argmin}{argmin}
\DeclareMathOperator{\argmax}{argmax}
\usepackage[capitalize,noabbrev]{cleveref}
\usepackage{amsthm}
\usepackage{amsmath}
\theoremstyle{plain}
\newtheorem{Th}{Theorem}[section]

\theoremstyle{definition}

\theoremstyle{remark}
\newtheorem{Rem}[Th]{Remark}
\crefname{Th}{Theorem}{Theorems}
\crefname{Prop}{Proposition}{Propositions}
\crefname{Lemma}{Lemma}{Lemmas}
\crefname{Cor}{Corollary}{Corollaries}
\crefname{Def}{Definition}{Definitions}
\crefname{Rem}{Remark}{Remarks}
\crefname{assumption}{Assumption}{Assumptions}

\Crefname{Th}{Theorem}{Theorems}
\Crefname{Prop}{Proposition}{Propositions}
\Crefname{Lemma}{Lemma}{Lemmas}
\Crefname{Cor}{Corollary}{Corollaries}
\Crefname{Def}{Definition}{Definitions}
\Crefname{Rem}{Remark}{Remarks}
\Crefname{assumption}{Assumption}{Assumptions}

\newcommand{\norm}[1]{\left\|#1\right\|}

\newcommand{\alglinelabel}{%
  \addtocounter{ALC@line}{-1}% Reduce line counter by 1
  \refstepcounter{ALC@line}% Increment line counter with reference capability
  \label% Regular \label
}

\newcommand{\yrcite}[1]{\citeyearpar{#1}}
\renewcommand{\cite}[1]{\citep{#1}}

\title{Iteration Complexity of Frank-Wolfe and Its Variants\\for Bilevel Optimization}
\author{Anthony Palmieri, Francesco Rinaldi, Saverio Salzo, Sara Venturini}
\date{}

\begin{document}

\maketitle

\begin{abstract}
%\noindent
We study Frank-Wolfe (FW) methods for constrained bilevel optimization when the lower-level problem is solved only approximately, yielding biased and inexact hypergradients. We analyze inexact variants of vanilla FW as well as away-step and pairwise FW, and provide convergence rates in the nonconvex setting under gradient errors. By combining these results with recent bounds on hypergradient errors from iterative and approximate implicit differentiation, we derive overall iteration-complexity guarantees for bilevel FW. Experiments on two real-world applications validate the theory and demonstrate practical effectiveness.
\end{abstract}

\section{Introduction}\label{bil}
Bilevel optimization represents a fundamental modeling paradigm in modern machine learning, with applications in hyperparameter optimization, meta-learning, neural architecture search, data reweighting \citep{franceschi2018bilevel, lorraine2020optimizing, pedregosa2016hyperparameter}. In these applications, the optimization variables are naturally separated into two levels: the upper-level variables, which often correspond to model or algorithmic parameters, and the lower-level variables, which typically represent model weights or subproblem solutions. The upper-level objective hence   implicitly depends on the solution of the lower-level problem, giving rise to a nested structure that is at the same time powerful and challenging to optimize.

Formally, bilevel problems can be expressed as follows:
\begin{equation} \label{bilevelproblem}
\begin{aligned}
    \min_{x \in \set} & \quad f(x) \coloneqq E(w(x), x) \\
    \text{s.t. } & \quad w(x) = \Phi(w(x), x) \,,
\end{aligned}
\end{equation}
where $\set$ is a compact convex subset of $\mathbb{R}^d$ and $E\colon\mathbb{R}^d \times \set \to \mathbb{R}$ and $\Phi\colon\mathbb{R}^d \times \set \to \mathbb{R}^d$ are continuously differentiable functions. 
We assume that the lower-level problem in~\eqref{bilevelproblem} (which is a parametric fixed point-equation) admits a unique solution. However, in general, explicitly computing such solution is either impossible or expensive. When $f$ is differentiable, this issue affects the evaluation of the gradient $\nabla f(x)$ which at best can only be approximately computed as $\tilde \nabla f(x)$.
A prototypical example of the bilevel problem~\eqref{bilevelproblem} is
\begin{equation} \label{bilevelproblem2}
\begin{aligned}
    \min_{x \in \set} & \quad f(x) \coloneqq E(w(x),x)\\
    \text{s.t. } & \quad w(x) = \argmin_{w \in \mathbb{R}^d} \ell(w,x)\,,
\end{aligned}
\end{equation}
where $\ell\colon\mathbb{R}^d \times \set \to \mathbb{R}$  is a loss function, twice continuously differentiable and strongly convex w.r.t. the first variable. Indeed, if we let $\Phi$ be such that $\Phi(w,x) = w - \alpha(x) \nabla_1 \ell(w, x)$, where $\alpha\colon \set \to \mathbb{R}_{++}$ is differentiable, then problem~\eqref{bilevelproblem2} and problem~\eqref{bilevelproblem} are equivalent.

%%\noindent
In \cite{grazzi2020iteration}, under certain assumptions -- namely that $\Phi(\cdot,x)$ is a contraction with constant $q_x \in (0,1)$ and that $\Phi$, $E$, and their derivatives are Lipschitz continuous -- the authors propose two procedures to compute an approximate gradient of the upper-level objective in \eqref{bilevelproblem}, referred to in the bilevel optimization literature as the \emph{hypergradient}, together with an analysis of their iteration complexity.
In particular, in Iterative Differentiation (ITD), the authors solve the lower-level problem using $t$ fixed-point iterations and compute the hypergradient via automatic differentiation. In Approximate Implicit Differentiation (AID), they also solve the lower-level problem using $t$ fixed-point iterations, but additionally solve a related linear system using $k$ iterations, and approximate the hypergradient using the outputs of these two procedures.

%%\noindent
Building on these results, the vast majority of bilevel optimization approaches combine such hypergradient approximations with projection-based first-order methods at the upper level. In many machine learning applications, however, the feasible set $\set$ may have a complex structure that makes the use of projection oracles computationally expensive, whereas linear minimization oracles—i.e., minimizing a linear function over $\set$—are typically cheaper to evaluate. \citet{combettes2021complexity} analyze complexity bounds for both types of oracles on several sets commonly used in real-world applications. In such settings, projection-free methods, and in particular the Frank–Wolfe (FW) algorithm and its variants, provide an attractive alternative to projection-based approaches (see, e.g., \cite{bomze2021frank, cgbook} for overviews).
However, to the best of our knowledge, despite the popularity of projection-free methods, their theoretical analysis in the bilevel optimization context remains largely unexplored.

\paragraph{Contributions.}
In this work, we provide a comprehensive complexity analysis for FW type methods in bilevel optimization, coupled with ITD and AID approximation schemes for the estimation of the hypergradient.
%when considering the scenario where only \emph{inexact hypergradients} are  available.
We hence aim to complement the results mentioned above to the case where the lower-level problem is solved inexactly, and the upper-level problem is addressed using the FW algorithm or one of its variants (i.e., away-step and pairwise FW). In particular, our main contributions can be summarized as follows:

\begin{itemize}
\item \textbf{Analysis of inexact FW  for nonconvex problems.} We establish an iteration complexity result for the classical Frank-Wolfe algorithm when the exact gradient is replaced by an inexact estimate. Using the FW gap as a stationarity measure, we show convergence to an $\mathcal{O}(\tau)$-stationary point in $\mathcal{O}(\tau^{-2})$ iterations, despite biased gradient errors.

\item \textbf{Analysis of inexact FW  variants for nonconvex problems.} We extend the analysis to the away-step and a pairwise variant with limited number of swaps on polytopal domains. These variants, which improve practical performances and sparsity of the iterates, need a more careful analysis under inexact gradients due to support changes and bad steps, i.e., drop/swap steps  for which we cannot show good progress. We prove that, when dealing with inexact gradients, both algorithms have an iteration complexity similar to that of the vanilla FW algorithm.

\item \textbf{Iteration complexity for bilevel optimization.} By combining our analysis of the inexact FW methods with existing hypergradient error bounds for ITD and AID \cite{grazzi2020iteration}, we derive iteration complexity guarantees for the constrained bilevel optimization setting. These bounds explicitly connect  lower-level iterations with the FW stationarity gap, giving an overall complexity bound of order $\mathcal{O}(\tau^{-2}\log(\tau^{-1}))$.

\end{itemize}

\section{Related Work}

In this section, we provide a brief review of gradient-based approaches for bilevel optimization, which currently represent the dominant paradigm in the literature. We then discuss projection-free methods, and in particular Frank–Wolfe-type algorithms, highlighting existing results and the gap that motivates our work.

\paragraph{Gradient-based methods in bilevel optimization.}
In recent years, gradient-based approaches to bilevel optimization have emerged as a practical and scalable framework for solving hierarchical learning and decision-making problems. In these methods, the function  $f$ in \eqref{bilevelproblem2} is minimized using an inexact projected gradient descent algorithm, which tolerates errors in the computation of the (hyper)gradient in order to accommodate inexact solutions of the lower-level problem.

%%\noindent
In this context, two major techniques are used to obtain effective estimates of the hypergradient: \emph{Iterative Differentiation} (ITD) \cite{maclaurin2015,franceschi2017,franceschi2018bilevel} and \emph{Approximate Implicit Differentiation} (AID) \cite{pedregosa2016hyperparameter,lorraine2020optimizing}. These two procedures are analyzed in detail by \citet{grazzi2020iteration}, who provide a comprehensive study of their iteration complexity.

%%\noindent
The convergence rate of the full gradient-based bilevel procedure for solving problems \eqref{bilevelproblem}–\eqref{bilevelproblem2}, based on either ITD or AID, was initiated by \citet{ghadimi2018}. They presented the first convergence analysis of a simple double-loop procedure, in both deterministic and stochastic settings. In the deterministic setting, which is the focus of this work, they established a convergence rate of $\mathcal{O}(\tau^{-5/4})$ for $\tau$-stationary points. More recent studies, such as \cite{Ji2021,arbel2021,grazzijmlr23}, have improved this result by proving the optimal iteration complexity of $\mathcal{O}(\tau^{-1})$.

\paragraph{Frank-Wolfe and projection-free methods.}
Over the last decade, FW methods have been the subject of  extensive studies, which gave numerous algorithmic variants with improved convergence guarantees. A modern analysis for the convex  case is established by  \cite{jaggi2013revisiting}. Away-step and pairwise variants were introduced to mitigate zig-zagging behavior and to achieve linear convergence under polyhedral constraints and strong convexity of the objective, see \cite{guelat1986some,lacoste2015global}. 

%\noindent
FW-type methods have also been analyzed for nonconvex objectives. Sublinear convergence to stationary points measured via the FW gap has been established in \citet{lacoste2016convergence} for vanilla FW and in \citet{bomze2020active} for the away-step variant.  Comprehensive modern surveys detailing theoretical advances and practical behavior of projection-free methods can be found in \citet{bomze2021frank} and in the seminal book by \citet{cgbook}.

%\noindent
Recent works have studied conditional gradient methods under gradient errors, providing convergence guarantees that depend on the magnitude and structure of the noise, see, e.g., \citet{venturini2023learning} and references therein for further details. However, these analyses do not directly address the bilevel setting where the lower-level problem is solved only approximately. Our work bridges this gap, thus giving explicit iteration complexity guarantees for projection-free methods in bilevel optimization.

\section{Frank-Wolfe algorithms with inexact gradients for nonconvex problems}
\label{sec:FWmethods}
In this section, we present three Frank-Wolfe-type methods that handle inexact gradient information. For those algorithms, we provide iteration complexity---all proofs are deferred to the Appendix~\ref{Aproofs}. We address the optimization problem
 \begin{equation}
 \label{bprob2}
 \begin{aligned}
  \min_{x\in C} f(x),
 \end{aligned}
 \end{equation}
 under the following assumption.
 \begin{assumption}
\label{assumptionbasic}
The set $C\subset \R^d$ is nonempty convex and compact with diameter $\Delta>0$.
The function $f\colon\R^d\to \R$ is differentiable with $\Lip$-Lipschitz continuous gradient, meaning that
\begin{equation*}
\forall\,x,y\in \R^d\colon \norm{\nabla f(x)- \nabla f(y)}\leq \Lip\norm{x-y}.
\end{equation*}
\end{assumption}
\noindent
In addition, we assume that at each point $x\in C$ an estimate $\widetilde\nabla f(x)$ of the gradient of $f$ at $x$ is available. Following Venturini et al. \yrcite{venturini2023learning}, we assume that the estimate $\widetilde \nabla f$ satisfies the following condition.
\begin{assumption}\label{assumption}
There exists $\sigma\in \left[0,1/3\right[$,
such that
\begin{equation}\label{cond_inexact}
|(\nabla f(\bar{x}) - \widetilde \nabla f(\bar{x}))^\top (x-\bar{x})| \leq \frac{\sigma}{1+\sigma}\tau \quad \forall\  x,\bar{x} \in \set,
\end{equation}
for some tolerance $\tau>0$.
% with $\epsilon_n \ge 0$.
\end{assumption}

%\noindent
The inexact Frank-Wolfe methods we consider will all produce a sequence $(x_n)_{n\in \N}$ in $C$ according to the iteration below
\begin{equation}
x_{n+1} = x_n + \eta_n d_n,
\end{equation}
where $\eta_n>0$ is the stepsize and $d_n\in \R^d$ is an appropriate direction
computed using the inexact gradient $\widetilde\nabla f(x_n)$.
For such sequence we set, for every $n\in \N$,
\begin{align}
\label{eq:gaps}
g_n &= -\nabla f(x_n)^\top d_n,\\
\tilde g_n &= -\widetilde\nabla f(x_n)^\top d_n,\\
g_n^{FW} &= -\nabla f(x_n)^\top d_n^{FW},
\end{align}
where $d_n^{FW} \in \argmin_{x \in \set}\{\nabla f(x_n)^\top (x-x_n)\} - x_n$ is the \emph{FW direction} at $x_n$. Since $\set$ is a closed convex set, a point $x^* \in \set$ is stationary for problem \eqref{bprob2} when 
\begin{equation*}
(\forall\,x\in C)\quad \nabla f(x^*)^\top (x-x^*) \ge 0.
\end{equation*}
Then, since $g^{FW}_n = - \min_{x\in C} \nabla f(x_n)^\top (x-x_n)\geq 0$,  $g^{FW}_n$ is indeed an optimality measure, i.e.,\ $g^{FW}_n=0$ if and only if $x_n\in \set$ is a stationary point. Finally, we introduce the best duality FW gap along the first $N$ iterations as 
\begin{equation*}
g_{N}^* = \min_{0 \leq i \leq N} g^{FW}_i
\end{equation*}
and we set $f^* = \min_{x \in \set} f(x)$.

\subsection{Vanilla inexact Frank-Wolfe Algorithm}\label{prel}

In this section, we present the basic FW algorithm  with inexact gradients.
In Algorithm \ref{alg:final_FW}, we report the detailed scheme. 
 Starting from an initial feasible point $x_0 \in \set$, the algorithm generates a sequence of feasible iterates using only linear minimizations over $\set$.
More specifically, at iteration $n$, an approximate gradient $\widetilde{\nabla} f(x_n)$ is computed  (line \ref{line:4}) and used to feed the linear minimization oracle~(LMO):
\[
\hat{x}_n \in \arg\min_{x \in \set} \; \widetilde{\nabla} f(x_n)^\top (x - x_n),
\]
which gives back the point $\hat x_n$  defining the inexact FW  direction $d_n = \hat{x}_n - x_n$ (See lines \ref{line:5}-\ref{line:6}). The next iterate is then  obtained through a convex combination
\[
x_{n+1} = x_n + \eta_n d_n,
\]
where the stepsize $\eta_n \in (0,1]$ is selected via a suitably chosen line search (lines \ref{line:7}-\ref{line:8}). The procedure then ends when the FW gap falls below a prescribed tolerance~$\tau$ (line~\ref{line:10}). 
%\noindent
The key distinction from the classical  method is that both the LMO and the stopping criterion  use inexact gradient information, which introduces a controlled bias to be explicitly handled in the convergence analysis.

\begin{algorithm}[tb]
\caption{Frank-Wolfe Algorithm with Inexact Gradient}
\label{alg:final_FW}
\begin{algorithmic}[1]
\STATE \textbf{Input:} initial point $x_0 \in \set$, tolerance $\tau > 0$
\STATE  Set $n:= 0$
\REPEAT
    \STATE Compute $\widetilde\nabla f(x_n)$, an estimate of $\nabla f(x_n)$  \alglinelabel{line:4}
    \STATE Compute $\hat{x}_n \in \arg\min_{x \in \set} \widetilde\nabla f(x_n)^\top (x - x_n)$
     \alglinelabel{line:5}
    \STATE Set $d_n := \hat{x}_n - x_n$\alglinelabel{line:6}
    \STATE Compute step size $\eta_n \in (0, 1]$ via line search \alglinelabel{line:7}
    \STATE Update $x_{n+1} := x_n + \eta_n d_n$ \alglinelabel{line:8}
    \STATE Set $n := n + 1$
\UNTIL{$-\widetilde\nabla f(x_{n-1})^\top d_{n-1} \leq \tau$} \alglinelabel{line:10}
\STATE Set $N := n-1$
\end{algorithmic}
\end{algorithm}
%\noindent
The main convergence result is presented below.
\begin{Th} \label{nonconvb}
Under Assumptions~\ref{assumptionbasic} and \ref{assumption},
suppose that in Algorithm~\ref{alg:final_FW}  
the step size $\eta_n$ satisfies the following conditions 
\begin{align}
\label{alphabound}
&{\eta}_n \ge \bar{\eta}_n = \min\left\{1, \frac{\tilde{g}_n}{{\Lip\n{d_n}^2}} \right\},\\[1ex]
\label{eq:rho}
&f(x_n) - f(x_n + \eta_n d_n) \geq \rho\bar{\eta}_n\tilde{g}_n,
\end{align}
with some fixed $\rho > 0$.
Then, Algorithm~\ref{alg:final_FW} terminates after a finite number of iterations, specifically in at most
% \begin{equation}
%     N = \left\lceil 
%     \max\left\{
%     \frac{\Delta^2 \Lip (f(x_0) - f^*) (1+\sigma)^2}{ \tau^2\rho (1 - \sigma)^2},
%     \frac{2 (f(x_0) - f^*) (1+\sigma)}{ \tau (1 - 3\sigma)}
%     \right\}
%     \right\rceil - 1 \,,
% \end{equation}
\begin{equation}
\label{eq:Nbound}
N = \left\lceil \max\{\alpha_1,\, \alpha_2\} \right\rceil - 1,
\end{equation}
where
\begin{equation}\label{eq:N1N2}
\begin{aligned}
\alpha_1 &= \frac{\Delta^2 \Lip \bigl(f(x_0)-f^*\bigr) (1+\sigma)^2}{\tau^2\rho (1-\sigma)^2},\\
\alpha_2 &= \frac{2 \bigl(f(x_0)-f^*\bigr) (1+\sigma)}{\tau (1-3\sigma)}.
\end{aligned}
\end{equation}
Moreover, for the best FW gap we have
	\begin{equation} \label{g_T}
		{g_{N}^* \leq \frac{\tau}{1+\sigma}} \,.
	\end{equation}
\end{Th}

%\noindent
%

\begin{Rem} 
Conditions \eqref{alphabound}-\eqref{eq:rho} can be satisfied when suitable line searches or stepsize rules are embedded in the method.
In \citet{venturini2023learning}, the authors proved that \eqref{alphabound}-\eqref{eq:rho} are satisfied using an Armijo line search; whereas, the next lemma (whose proof is given in Appendix~\ref{Aproofs}) shows that when the exact stepsize 
\[
\bar\eta_n = \min\!\left\{\eta_n^{max}, \frac{\tilde g_n}{\Lip \|d_n\|^2}\right\}
\; 
\]
is employed, the same conditions are also satisfied.  
\end{Rem}
\par\medskip
\begin{lemma}\label{lemma:es}
Under Assumptions~\ref{assumptionbasic} and \ref{assumption}, the stepsize rule $\eta_n = \bar\eta_n$ satisfies  conditions \eqref{alphabound}-\eqref{eq:rho}.
\end{lemma}

\begin{Rem} 
In the context of proximal-gradient methods for nonconvex problems, rate of convergence is often given in terms of the so-called proximal-gradient mapping, which for problem \eqref{bprob2} is defined as $\|\pi(x_n - \tilde \nabla f(x_n)) - x_n\|^2$. In \citet{bomze2021frank}, Equation (37), the authors proved that 
\begin{equation}\label{in:stat}
\|\pi(x_n -  \nabla f(x_n)) - x_n\|^2\leq g_n^{FW}.
\end{equation}
This provides a way to compare our convergence results with the ones available in the literature for proximal-gradient approaches. 

\end{Rem}

%\section{Frank-Wolfe Away-Step Algorithm (non convex) with Inexact Gradient}
\subsection{Inexact Away-Step Frank-Wolfe Algorithm}

In this section, we assume that  $\set$ is a convex polytope in $\mathbb{R}^d$. Thus,  $\set$ can be written as the convex hull of a finite set of points 
 \begin{equation}
\set \vcentcolon = \conv(\Ve), \quad \text{with}\quad \Ve= \{\ve_0,\ve_1,\ldots,\ve_q\} \subset \mathbb{R}^d.
 \end{equation}
This means that every point $x \in \set$ can be expressed as a convex combination of the vertices $\ve_j$'s
\begin{equation}
x = \sum_{j=0}^q w^j \ve_j, \quad \text{with } w^j \geq 0 \text{ and } \sum_{j=0}^q w^j = 1.
\end{equation}
Moreover, for every $u\in \mathbb{R}^d$, the linear optimization problems
\begin{equation*}
    \min_{x\in \set} u^\top x\quad\text{and}\quad \max_{x\in \set} u^\top x,
\end{equation*}
admit solutions among the vertices $\ve_1, \dots, \ve_q$. 

We now define the algorithm which will keep trace of the representation of the iterates with respect to the vertices $\ve_0, \dots, \ve_q$. More specifically, for every $n\in \mathbb{N}$, we have
\begin{equation*}
x_n = \sum_{j=0}^q w^j_n \ve_j,
\end{equation*}
and we will define the support of $(w^j_n)_{0\leq j\leq q}$ as $\supp_n = \{j\in \{0,\dots, q\}\,\vert\, w^j_n>0\}$, so that we essentially have
\begin{equation*}
x_n = \sum_{j\in \supp_n} w_n^j \ve_j.
\end{equation*}
%\noindent
In Algorithm~\ref{alg:SAFW}, we report the detailed scheme. Starting from an initial vertex $x_0 \in \Ve$, the algorithm maintains an explicit convex decomposition of the iterates in terms of the polytope vertices  $\Ve = \{\ve_0,\dots,\ve_q\}$. At iteration $n$, an approximate gradient $\widetilde{\nabla} f(x_n)$ is computed (line \ref{line:5AS}) and used to define two search directions. First, the inexact FW direction is obtained  $h_n = \hat{\ve}_i - x_n$ (lines \ref{line:6AS}-\ref{line:FW}). Second, we identify the vertex $\hat{\ve}_j$ in the current active set $\ve_n$ that maximizes the linearized objective,
\[
\ve_{\hat\j} \in \argmax_{\ve \in \conv(\{\ve_i\,\vert\,j\in \supp_n\}}) \widetilde{\nabla} f(x_n)^\top \ve,
\]
and define the corresponding inexact \emph{away-step direction}, which moves away from $\ve_{\hat\j}$, as $b_n = x_n - \ve_{\hat\j}$ (Lines \ref{line:8AS}-\ref{line:activeset}).
%\noindent
The algorithm then compares the potential decrease provided by the inexact FW and away-step directions and selects the best one (lines \ref{line:10AS}-\ref{line:14AS}). A step size $\eta_n>0$ is chosen by using a suitable line search within a feasible interval that preserves the convex combination structure (line~\ref{line:15AS}), and the iterate and weights are updated accordingly (lines~\ref{line:activeset_h1_start}-\ref{line:activeset_h1_end}). The procedure terminates when the inexact FW gap falls below a tolerance $\tau$ (line \ref{line:26AS}).

%\noindent
The key difference from the vanilla FW method is the use of away steps, which allow the method to reduce the weight of previously selected vertices and mitigate zig-zagging behavior. Same as for the vanilla FW, both the direction choice and the stopping criterion make use of inexact gradient information, requiring a refined convergence analysis.
\begin{algorithm}[t]
\caption{Away-Step Frank-Wolfe algorithm }
\label{alg:SAFW}
\small
\begin{algorithmic}[1]
\STATE \textbf{Input:} initial point $x_0\in \Ve=\{\ve_0,\ve_1,\ldots,\ve_q\} \subset \mathbb{R}^d$, tolerance $\tau > 0$.
\STATE Set $\supp_0 = \{i_0\}$, where $x_0 = \ve_{i_0}$ and, for every $j=0,\dots, q$, $w_0^j = \delta_{j, i_0}$.
\STATE Set $n \gets 0$
\REPEAT
\STATE \hspace*{0.2truecm}Compute $\widetilde \nabla f(x_n)$ as an estimate of $\nabla f(x_n)$ \alglinelabel{line:5AS}
\STATE \hspace*{0.2truecm}Let  
$\hat\i \in \argmin_{i \in \{0,\dots, q\}} \widetilde \nabla f(x_n)^\top \ve_i$ 
\alglinelabel{line:6AS}
\STATE \hspace*{0.2truecm}Set $h_n \vcentcolon= \ve_{\hat\i} - x_n$ (inexact FW direction) \alglinelabel{line:FW} 
\STATE \hspace*{0.2truecm}Let 
$\hat\j \in  \argmax_{j \in \supp_n} \widetilde \nabla f(x_n)^\top \ve_j$,
\alglinelabel{line:8AS}
\STATE \hspace*{0.2truecm}Set $b_n \vcentcolon=  x_n - \ve_{\hat\j}$  (inexact away-step direction)
\alglinelabel{line:activeset}
\STATE \hspace*{0.2truecm}\textbf{If} $-\widetilde\nabla f(x_n)^\top h_n \geq -\widetilde\nabla f(x_n)^\top b_n$ \textbf{then}
\alglinelabel{line:10AS}
\STATE \hspace*{0.2truecm}\hspace*{0.2truecm} $d_n \vcentcolon= h_n$ and $\eta_n^{\max} \vcentcolon= 1$ \phantom{$2^{\hat{H}}$}\alglinelabel{line:20260122a}
\STATE \hspace*{0.2truecm}\textbf{Else} (necessarily $w_n^{\hat\j}<1$, see Remark~\ref{rem:supportS}\ref{rem:support\ve_0})\alglinelabel{line:20260121a}
\STATE \hspace*{0.2truecm}\hspace*{0.2truecm} $d_n \vcentcolon= b_n$ and  
$\eta_n^{\max} := w_n^{\hat\j}/(1-w_n^{\hat\j})\,$\\
\hspace*{0.1truecm}$(\eta_n^{\max}=\max \{ \eta>0 \,\vert\, x_n + \eta b_n  \in \conv(\{\ve_j\,\vert\, j\in \supp_n\}))$\alglinelabel{line:stepsize}
\STATE \hspace*{0.2truecm}\textbf{End if} 
\alglinelabel{line:14AS}
\STATE \hspace*{0.2truecm}Compute a stepsize $\eta_n\in(0,\eta_n^{\max}]$ by a line search
\alglinelabel{line:15AS}
\STATE \hspace*{0.2truecm}Set $x_{n+1} \vcentcolon= x_n + \eta_n d_n$ 
\STATE \hspace*{0.2truecm}\textbf{If} $-\widetilde\nabla f(x_n)^\top h_n \geq -\widetilde\nabla f(x_n)^\top b_n$ \textbf{then}\alglinelabel{line:activeset_h1_start}
\STATE \hspace*{0.2truecm}\hspace*{0.2truecm} for every $j\neq \hat\i\colon$ $w_{n+1}^j := (1-\eta_n) w_n^j$ \alglinelabel{line:activeset_h1}
\STATE \hspace*{0.2truecm}\hspace*{0.2truecm} for $j=\hat\i\colon$ $w_{n+1}^j:= \eta_n$ if $\hat\i\notin \supp_n$, and\\ 
\hspace*{1.95truecm}$w_{n+1}^j:= (1-\eta_n)w_{n}^{\hat\i}+\eta_n$ if $\hat\i \in  \supp_n$ \alglinelabel{line:activeset_h2}
\STATE \hspace*{0.2truecm}\textbf{Else}
\STATE \hspace*{0.2truecm}\hspace*{0.2truecm} for every $j\neq \hat\j\colon$ $w_{n+1}^j = (1+\eta_n) w_n^j$ \alglinelabel{line:activeset_b1}
\STATE \hspace*{0.2truecm}\hspace*{0.2truecm} for $j=\hat\j\colon$ $w_{n+1}^j= (1+\eta_n)w_{n}^{\hat\j}-\eta_n$ \alglinelabel{line:activeset_b2}
\STATE \hspace*{0.2truecm}\textbf{End if}\alglinelabel{line:activeset_h1_end}
\STATE \hspace*{0.2truecm}Set $\supp_{n+1} := \{j\in \{0,\dots, q\}\,\vert\, w_{n+1}^j>0\}$.
\STATE \hspace*{0.2truecm}Set $n \gets n + 1$ 
\alglinelabel{line:24AS}
\UNTIL{$-\widetilde\nabla f(x_{n-1})^\top h_{n-1} \leq \tau$} 
\alglinelabel{line:26AS}
\STATE Set $N \gets n-1$
\end{algorithmic}
\end{algorithm}
We provide below a few remarks to facilitate the understanding of the theoretical analysis.

\newpage
\begin{Rem}\label{rem:supportS}\
\vspace{-1ex}
\begin{enumerate}[label={\rm (\roman*)}]
\item \label{rem:support\ve_0} Referring to line~\ref{line:20260121a} in Algorithm~\ref{alg:SAFW}, 
we note that since
\begin{align*}
\widetilde{\nabla} f(x_n)^\top h_n &= \widetilde{\nabla} f(x_n)^\top (\ve_{\hat\i}-x_n)\\
&= \min_{0\leq i\leq q} \widetilde{\nabla} f(x_n)^\top (\ve_{i}-x_n)\\ &= \min_{x\in C} \tilde{\nabla} f(x_n)^\top (x-x_n)\leq 0,
\end{align*}
we have that if $-\widetilde{\nabla} f(x_n)^\top h_n<-\widetilde{\nabla} f(x_n)^\top b_n$, then necessarily $b_n\neq 0$ and hence $\ve_{\hat\j}\neq x_n\ \Rightarrow\ \ve_{n}\neq \{\hat\j\}\ \Rightarrow\ w_n^{\hat\j}<1$.
\item\label{rem:support\ve_i}
In Algorithm~\ref{alg:SAFW},
the update of the weights in lines \ref{line:activeset_h1}-\ref{line:activeset_h2} 
comes from the fact that
\begin{align*}
x_{n+1} &= x_n + \eta_n(\ve_{\hat \i}-x_n) = (1-\eta_n)x_n + \eta_n\ve_{\hat\i}\\
&= \sum_{j\in \supp_n} (1-\eta_n) w_n^j + \eta_n\ve_{\hat\i},
\end{align*}
while that in lines \ref{line:activeset_b1}-\ref{line:activeset_b2}
 comes from the fact that
\begin{align*}
x_{n+1} &= x_n + \eta_n(x_n- \ve_{\hat \j}) = (1+\eta_n)x_n - \eta_n \ve_{\hat\j}\\
&= \sum_{j\in \supp_n\setminus\{\hat\j\}} (1+\eta_n) w_n^j + ((1+\eta_n)w_n^{\hat\j} - \eta_n)\ve_{\hat\j}.
\end{align*}
\item\label{rem:support\ve_ii} If $-\widetilde\nabla f(x_n)^\top h_n \geq -\widetilde\nabla f(x_n)^\top b_n$,
it follows from lines \ref{line:activeset_h1}-\ref{line:activeset_h2} 
in Algorithm~\ref{alg:SAFW} that, 
\begin{equation*}
\supp_{n+1}= 
\begin{cases}
\supp_n\cup \{\hat\i\} &\text{if}\ \eta_n<1 (=\eta_n^{\max})\\
 \{\hat\i\} &\text{if}\ \eta_n=1 (=\eta_n^{\max}),
\end{cases}
\end{equation*}
where clearly, when $\eta_n<1$ and $\hat{\i}\in \supp_n$, we have $\supp_{n+1}=\supp_n$.
On the other hand if $-\widetilde\nabla f(x_n)^\top h_n < -\widetilde\nabla f(x_n)^\top b_n$, if follows from lines \ref{line:activeset_b1}-\ref{line:activeset_b2} 
in Algorithm~\ref{alg:SAFW} that
\begin{equation*}
\supp_{n+1}= 
\begin{cases}
\supp_n &\text{if}\ \eta_n<\eta_n^{\max}\\
\supp_n\!\setminus\!\{\hat\j\} &\text{if}\ \eta_n=\eta_n^{\max}.
%\supp_n=\{\hat\j\} &\text{if}\ \eta_n^{\max}=1.
\end{cases}
\end{equation*}
\end{enumerate}
\end{Rem}

%\noindent
The main result of this section is presented below.
\begin{Th} \label{nonconvb-away}
Under Assumptions~\ref{assumptionbasic} and \ref{assumption},
suppose that~$C$ is a polytope and
 in Algorithm~\ref{alg:SAFW}  
the step size $\eta_n$ satisfies the following conditions 
 \begin{equation}    \label{alphabound2}
	{{\eta}_n \ge \bar{\eta}_n = \min\left\{\eta_n^{\max}, \frac{\tilde{g}_n}{{\Lip\n{d_n}^2}} \right\} \,,}
\end{equation}
	\begin{equation} \label{eq:rho2}
	f(x_n) - f(x_n + \eta_n d_n) \geq \rho\bar{\eta}_n\tilde{g}_n \,.
	\end{equation}
% with some fixed $\rho > 0$.
% Let $\{x_n\}$ be a sequence generated by Algorithm~\ref{alg:final_FW}, where $\widetilde \nabla f$ satisfies Assumption~\ref{assump:cond_inexact} with 
% \begin{equation}
% \label{assumption}
%\epsilon_n \leq \sigma \min(g_n,\tilde{g}_n), \quad \sigma < \frac{1}{3}
% {\epsilon_n \leq \frac{\sigma}{1+\sigma}\,  \tilde g_n, \quad 0 \le \sigma < \frac{1}{3},}
%  \end{equation} 
% and  the step size $\eta_n$ satisfies
%  \begin{equation}    \label{alphabound}
% 	{{\eta}_n \ge \bar{\eta}_n = \min\left(\eta_n^{\max}, \frac{\tilde{g}_n}{{\Lip\n{d_n}^2}} \right),}
% \end{equation}
%	with some fixed $\rho > 0$. 
Then, Algorithm~\ref{alg:SAFW} terminates after a finite number of iterations, specifically 
\begin{equation}
\label{eq:maxiter}
    N \leq \left\lceil 
    2\max\left\{
    \alpha_1, \alpha_2
    \right\}
    \right\rceil - 1 \,,
\end{equation}
with \(\alpha_1, \, \alpha_2\) as defined in \cref{eq:N1N2}, and for the best FW gap we have
	\begin{equation} \label{g_T2}
		{g_{N}^* \leq \frac{\tau}{1+\sigma}} \,.
	\end{equation}
\end{Th}

\subsection{Inexact Pairwise Frank-Wolfe algorithm with controlled swaps}

In this section, we introduce an inexact version of the Pairwise Frank-Wolfe (PWFW) Algorithm.
However, to give a converge rate on the full iterate sequence, including the so-called bad steps 
that are normally discarded in the literature (see, e.g., \citet{lacoste2015global}), we slightly modify the algorithm introducing a cap on the maximum number of consecutive swaps that can be performed.

%\noindent
Algorithm~\ref{alg:PWFW} details the procedure. As for the away-step variant, the feasible region $\set$ is assumed to be a polytope, and each iterate $x_n$ is represented as a convex combination of its vertices with active set $\ve_n$.
At each iteration $n$, inexact FW and away-step directions are calculated (lines \ref{line:6PW}-\ref{line:activeset2}) and then combined to obtain the inexact pairwise direction, i.e., a direction that moves weight from the away-step vertex $\hat{\ve}_j$ to the FW vertex $\hat{\ve}_i$ (line \ref{line:stepsize2}).  In order to control the number of consecutive swap steps, the algorithm updates a counter and eventually switches to an away step when a prescribed threshold is reached (lines \ref{line:13PW}-\ref{line:24PW}). The iterate and weights are updated accordingly, and the process stops when the inexact FW gap is below the selected precision $\tau$ (line \ref{line:27PW}).

%\noindent
Compared to the other FW algorithms, the pairwise variant performs direct weight transfers between vertices, often leading to sparser iterates and improved practical behavior. As in the previous algorithms, both direction selection and termination step use inexact gradients, and the resulting gradient bias must be explicitly handled in the convergence guarantees.
\begin{algorithm}[tb]
\caption{Pairwise Frank-Wolfe algorithm }
\label{alg:PWFW}
\small
\begin{algorithmic}[1]
\STATE \textbf{Input:} initial point $x_0\in \Ve=\{\ve_0,\ve_1,\ldots,\ve_q\} \subset \mathbb{R}^d$, tolerance $\tau > 0$, $R\in \N$.
\STATE Set $\supp_0 = \{i_0\}$, where $x_0 = \ve_{i_0}$ and, for every $j=0,\dots, q$, $w_0^j = \delta_{j, i_0}$.
\STATE Set $r \gets 0$
\STATE Set $n \gets 0$
\REPEAT
\STATE \hspace*{0.2truecm}Compute $\widetilde \nabla f(x_n)$ as an estimate of $\nabla f(x_n)$ 
\alglinelabel{line:6PW}
\STATE \hspace*{0.2truecm}Let  
$\hat\i \in \argmin_{i \in \{0,\dots, q\}} \widetilde \nabla f(x_n)^\top \ve_i$ \STATE \hspace*{0.2truecm}Set $h_n \vcentcolon= \ve_{\hat\i} - x_n$ (inexact FW direction) \alglinelabel{line:FW2} 
\STATE \hspace*{0.2truecm}Let 
$\hat\j \in  \argmax_{j \in \supp_n} \widetilde \nabla f(x_n)^\top \ve_j$,
\alglinelabel{line:10PW}
\STATE \hspace*{0.2truecm}Set $b_n \vcentcolon=  x_n - \ve_{\hat\j}$  (inexact away-step direction)
\alglinelabel{line:activeset2}
\STATE \hspace*{0.2truecm} $d_n \vcentcolon= h_n+b_n= \ve_{\hat\i}- \ve_{\hat\j}$ and  
$\eta_n^{\max} := w_n^{\hat\j}$\\
\hspace*{0.1truecm}$(=\max \{ \eta\in \left(0,1\right] \,\vert\, x_n + \eta d_n  \in \conv(\{\ve_j\,\vert\, j\in \supp_n\}))$
\alglinelabel{line:stepsize2}
\STATE \hspace*{0.2truecm}Compute a stepsize $\eta_n\in(0,\eta_n^{\max}]$ by a line search
\STATE \hspace*{0.2truecm}\textbf{If} $\eta_n =\eta_n^{max}$ and $\hat\i \notin \ve_n$ (swap step, see Remark~\ref{rmk:pairwiseFW})
\alglinelabel{line:13PW}
\STATE \hspace*{1truecm} $r = r + 1$
\STATE \hspace*{0.2truecm}\textbf{End if}
\STATE \hspace*{0.2truecm}\textbf{If} $r \leq R$ (Pairwise step)
\STATE \hspace*{0.5truecm}Set $x_{n+1} \vcentcolon= x_n + \eta_n d_n$ \alglinelabel{line:pairwise_start}
\STATE \hspace*{0.5truecm}Set for every $j\in \supp_n\setminus\{\hat\i,\hat\j\}\colon w_{n+1}^j := w_n^j$ \alglinelabel{line:activeset_h1}
\STATE \hspace*{0.5truecm}Set for $j=\hat\i\colon w_{n+1}^j:= \eta_n$ if $\hat\i\notin \supp_n$,  and \\ \hspace*{2.4truecm}$w_{n+1}^j:= w_{n}^{\hat\i}+\eta_n$ if $\hat\i \in  \supp_n$ \alglinelabel{line:activeset_h2}
\STATE \hspace*{0.5truecm}Set for $j=\hat\j\colon w_{n+1}^j:= w_n^{\hat\j}-\eta_n$\alglinelabel{line:pairwise_end}
\STATE \hspace*{0.2truecm}\textbf{Else}\alglinelabel{awaystep_start}
\STATE \hspace*{0.5truecm}Away-Step Frank-Wolfe lines 8--21 in Algorithm~\ref{alg:SAFW}
\STATE \hspace*{0.5truecm}$r = 0$
\STATE \hspace*{0.2truecm}\textbf{End if}
\alglinelabel{line:24PW}
\STATE \hspace*{0.2truecm}Set $\supp_{n+1} := \{j\in \{0,\dots, q\}\,\vert\, w_{n+1}^j>0\}$.
\STATE \hspace*{0.2truecm}Set $n \gets n + 1$
\UNTIL{$-\widetilde\nabla f(x_{n-1})^\top d_{n-1} \leq \tau$ }
\alglinelabel{line:27PW}
\STATE Set $N \gets n-1$
\end{algorithmic}
\end{algorithm}
In the following, 
%\noindent
we report a few remarks to facilitate the understanding of the theoretical analysis. 

%\noindent
\begin{Rem}\label{rmk:pairwiseFW}\ 
\begin{enumerate}[label={\rm (\roman*)}]
% \vspace{-2.6em}
\item In Algorithm~\ref{alg:PWFW},
the update of the weights in \ref{line:pairwise_start}-\ref{line:pairwise_end} 
comes from the fact that
\begin{align*}
x_{n+1} &= x_n + \eta_n(\ve_{\hat \i}-\ve_{\hat\j})\\
&= \sum_{j\in \supp_n\setminus\{\hat\i,\hat\j\}} w_n^j \ve_j + (w_n^{\hat\i}+\eta_n) \ve_{\hat\i} + (w_n^{\hat\j}-\eta_n) \ve_{\hat\j} .
\end{align*}
\item
It follows from lines \ref{line:pairwise_start}-\ref{line:pairwise_end} 
of Algorithm~\ref{alg:PWFW}, that
\begin{equation*}
\supp_{n+1}= 
\begin{cases}
\supp_n\cup\{\hat\i\} &\text{if}\ \eta_n<\eta_n^{\max}\ \text{and}\ \hat\i\not\in \supp_n\\
\supp_n &\text{if}\ \eta_n<\eta_n^{\max}\ \text{and}\ \hat\i\in \supp_n\\
(\supp_n\!\setminus\!\{\hat\j\})\cup\{\hat\i\} &\text{if}\ \eta_n=\eta_n^{\max}\ \text{and}\ \hat\i\not\in \supp_n\\
\supp_n\!\setminus\!\{\hat\j\} &\text{if}\ \eta_n=\eta_n^{\max}\ \text{and}\ \hat\i\in \supp_n
\end{cases}
\end{equation*}
It is clear that when $\eta_n=\eta_n^{\max}$ and $\hat\i\not\in \supp_n$, we have a swap in the support of the weights and   $\abs{\supp_{n+1}} = \abs{\supp_{n}}$. Moreover, if $\eta_n=1$, then $\eta_n^{\max}=w_n^{\hat\j}=1\ \Rightarrow\ \supp_n=\{\hat\j\}$ and $w_{n+1}^{\hat\j}=0$, so that necessarily $\supp_{n+1}= \{\hat\i\}$.
\item Algorithm~\ref{alg:PWFW} performs always a pairwise FW step. However, when $r=R$
and $\eta_n =\eta_n^{max}$ and $\hat\i \notin \ve_n$, it enters in the else branch 
at line 21
%\ref{awaystep_start}
and an away step is performed (with $\eta_n= \eta_n^{\max}=w_n^{\hat\j}$ and $\hat\i\not\in \supp_n$).
In this case $\eta_n^{\max}$ is recomputed according to lines 
\ref{line:20260122a}-\ref{line:stepsize}
in Algorithm~\ref{alg:SAFW}.
\end{enumerate}
\end{Rem}

%\noindent
We now report the main result of this section.
\begin{Th} \label{nonconvbp-pairwise}
Under Assumptions~\ref{assumptionbasic} and \ref{assumption},
suppose that $C$ is a polytope and
 in Algorithm~\ref{alg:SAFW}  
the step size $\eta_n$ satisfies the following conditions 
 \begin{equation}    \label{alphabound2p}
	{{\eta}_n \ge \bar{\eta}_n = \min\left\{\eta_n^{\max}, \frac{\tilde{g}_n}{{\Lip\n{d_n}^2}} \right\} \,,}
\end{equation}
	\begin{equation} \label{eq:rho2p}
	f(x_n) - f(x_n + \eta_n d_n) \geq \rho\bar{\eta}_n\tilde{g}_n \,.
	\end{equation}
% with some fixed $\rho > 0$.
% Let $\{x_n\}$ be a sequence generated by Algorithm~\ref{alg:final_FW}, where $\widetilde \nabla f$ satisfies Assumption~\ref{assump:cond_inexact} with 
% \begin{equation}
% \label{assumption}
%\epsilon_n \leq \sigma \min(g_n,\tilde{g}_n), \quad \sigma < \frac{1}{3}
% {\epsilon_n \leq \frac{\sigma}{1+\sigma}\,  \tilde g_n, \quad 0 \le \sigma < \frac{1}{3},}
%  \end{equation} 
% and  the step size $\eta_n$ satisfies
%  \begin{equation}    \label{alphabound}
% 	{{\eta}_n \ge \bar{\eta}_n = \min\left(\eta_n^{\max}, \frac{\tilde{g}_n}{{\Lip\n{d_n}^2}} \right),}
% \end{equation}
%	with some fixed $\rho > 0$. 
Then, Algorithm~\ref{alg:PWFW} terminates after a finite number of iterations, specifically in at most 
\begin{equation}
\label{eq:alg:PWFW}
        N \leq \left\lceil 
    2(R+1)\max\left\{
    \alpha_1, \alpha_2
    \right\}
    \right\rceil - 1 \,,
\end{equation}
with \(\alpha_1, \, \alpha_2\) as defined in \cref{eq:N1N2}, and for the best FW gap we have
	\begin{equation} \label{g_T2p}
		{g_{N}^* \leq \frac{\tau}{1+\sigma}} \,.
	\end{equation}
\end{Th}
% %$g_n^{FW}\to 0$ as $n\to \infty$. In particular, %for every $T \in \mathbb{N}$ we have
% 	\begin{equation} \label{g_T}
% 		{g_n^* \leq \max\left(\sqrt{\frac{2\Delta^2 \Lip (f(x_0) - f^*)}{n \rho (1 - \sigma)^2}}, \frac{4(f(x_0) - f^*)}{n(1 - 3\sigma)} \right),}
% 	\end{equation}
%  where $\displaystyle g^*_n = \min_{0 \leq i \leq n-1} g^{FW}_i$ 
%  and $f^* = \min_{x \in S} f(x)$.

\section{Convergence Result for the Bilevel Problem}
\label{sec:bilevelconvergence}

In this section we recall the iterative and approximate implicit differentiation techniques and the related approximation results. Details are given in \cite{grazzi2020iteration}.
Then, we give the overall iteration complexity for solving the bilevel problem \eqref{bilevelproblem} by the Frank-Wolfe algorithm and its variants previously described. 

%\noindent
The Iterative differentiation procedure (ITD) performs $t$ fixed-point iterations as follows.
Let $w_0(x)=0$ and set
\begin{equation}
\label{eq:ITD}
\begin{array}{l}
    \text{for}\;i=1,2,\ldots t\\[0.4ex]
    \left\lfloor
    \begin{array}{l}
    w_i(x) = \Phi(w_{i-1}(x), x).
    \end{array}    
    \right.
\end{array}
\vspace{-1ex}
\end{equation}
Setting $f_t(x) = E(w_t(x), x)$,
the estimate $\widetilde \nabla f(x)$ is then defined as $\widetilde \nabla f(x)=\nabla f_t(x)$, where  $\nabla f_t(x)$ is computed using automatic differentiation.
The corresponding approximation guarantees is
\begin{equation*}
\lVert\widetilde\nabla f(x)- \nabla f(x)\rVert \leq (\const_1(x)+\const_3(x)) q_x^t + \const_2(x)t q_x^{t-1},
\end{equation*}
where $\const_1(x), \const_2(x)$ and $\const_3(x)$ can be uniformly bounded by a constant $\Const$ depending on the Lipschitz constants of $\Phi$
and $E$ and their derivatives, as well as the diameter $\Delta$ of $C$. This implies that, for every $x, \bar{x}\in \set$,
\begin{equation*}
\abs{(\nabla f(\bar{x}) - \widetilde \nabla f(\bar{x}))^\top (x-\bar{x})} \leq \Const(2q^t + t q^{t-1}) \Delta,
\end{equation*}
where $q = \max_{x\in \set} q_x<1$.
Therefore, in order to fulfill Assumption~\ref{assumption}, we need to request that
\begin{equation*}
\Const(2q^t + t q^{t-1}) \Delta \leq \frac{\sigma}{1+\sigma}\tau.
\end{equation*}
This leads\footnote{If $\varepsilon\in \left]0,1\right[$, then
$\displaystyle t q^{t-1} = q^{-1} t q^t =  (\varepsilon q)^{-1}\varepsilon t q^{\varepsilon t} q^{(1-\varepsilon)t} \leq (\varepsilon q)^{-1}(e \log q^{-1})^{-1}q^{(1-\varepsilon)t}$, where we used $s q^s \leq (e \log q^{-1})$.} to
\begin{equation*}
t \geq \left\lceil \frac{1}{1-\varepsilon}\displaystyle\log_q\left(\displaystyle\frac{\sigma \tau}{(1+\sigma)\Delta \Const} \cdot\frac{\varepsilon q \log q^{-1}}{1+2\varepsilon q \log q^{-1}}\right) \right\rceil,
\end{equation*}
for any $\varepsilon\in \left]0,1\right[$.

%\noindent
On the other hand, the Approximate Implicit Differentiation (AID) technique (with fixed point method)
after performing the procedure \eqref{eq:ITD}, it solves approximately the linear system
\begin{equation}
(I - \partial_1 \Phi(w_t(x), x)^\top) u = \nabla_1 E(w_t(x),x)
\end{equation}
by performing $k$ fixed-point iterations as follows
\begin{equation*}
 \begin{array}{l}
\text{for}\;j=1,2,\ldots k\\[0.4ex]
\left\lfloor
\begin{array}{l}
u_{t,j}(x) =    \partial_1\Phi(w_t(x), x)^\top u_{t,j-1}(x) 
+\! \nabla_1 E(w_t(x), x),
\end{array}    
\right.
\end{array}
\end{equation*}
where $u_{t,0}=0$.
The estimate of the gradient of $f$ at $x$ is computed as
$\widetilde\nabla f(x) = \nabla_2 E(w_t(x), x) + \partial_2 \Phi(w_t(x),x)^\top u_{t,k}(x)$.
The guarantees now are as follows
\begin{equation*}
\lVert\widetilde\nabla f(x)- \nabla f(x)\rVert \leq \bigg(\const_1(x) + \frac{\const_2(x)}{1-q_x}\bigg) q_x^t + \const_3(x) q_x^k,
\end{equation*}
where $\const_1(x)$, $\const_2(x)$, and $\const_3(x)$ are as above.
%\noindent 
Again, defining $\Const$ and $q$ as in the previous case, it is clear that  to fulfill Assumption~\ref{assumption}, we can require
\begin{equation*}
     \left( \left( \Const + \frac{\Const}{1-q}\right) q^t + \Const q^k \right)
     \Delta 
     \leq
     \frac{\sigma}{1+\sigma}\tau\,.
\end{equation*}
Assuming $k=t$, this implies that in AID, we need to compute
a number of iterations 
\begin{equation}
\label{numiter2}
    k = t \geq \left\lceil \displaystyle\log_q\left(\displaystyle\frac{\sigma \tau}{(1+\sigma)\Delta \Const} \cdot\frac{1-q}{3-2q}\right) \right\rceil.
\end{equation}

%\noindent
In conclusion, we can put together the above iteration complexity estimates with the results given in Section~\ref{sec:FWmethods}.
To solve the bilevel problem~\eqref{bilevelproblem} using the  (ITD) or (AID) approximation techniques for the gradient of $f$ and any of the three Frank-Wolfe algorithms described in Section~\ref{sec:FWmethods} for the upper-level problem, the total number of (upper and lower) iterations is bounded by the following quantity
\begin{equation*}
N\times
\begin{cases}
\displaystyle
\left\lceil \frac{1}{1-\varepsilon}\displaystyle\log_q\left(\displaystyle\frac{\sigma \tau}{(1+\sigma)\Delta \Const} \cdot\frac{\varepsilon q \log q^{-1}}{1+2\varepsilon q \log q^{-1}}\right) \right\rceil\\[3ex]
\left\lceil \displaystyle\log_q\left(\displaystyle\frac{\sigma \tau}{(1+\sigma)\Delta \Const} \cdot\frac{1-q}{3-2q}\right) \right\rceil,
\end{cases}
\end{equation*}
for (ITD) and (AID), respectively,
with $N$ as in \eqref{eq:Nbound}, \eqref{eq:maxiter}, or \eqref{eq:alg:PWFW} depending on the chosen algorithm. The best Frank-Wolfe gap is less or equal than $\tau / (1+\sigma)$. Asymptotically the iteration complexity is of the order $\mathcal{O} (\tau^{-2} \log(\tau^{-1}) )$.

\section{Illustrative Experiments}
We report experiments on two constrained bilevel learning tasks, chosen to highlight the behavior of FW methods under inexact hypergradients.
We compare vanilla FW with its structured variants Away-Step FW (ASFW) and Pairwise FW (PWFW). In all cases, hypergradients are computed from an approximate solution of the lower-level problem.

%\noindent
We first consider the multilayer semi-supervised learning bilevel model of \citet{venturini2023learning}, where the upper level minimizes a validation multiclass cross-entropy loss and the lower level solves a Laplacian-regularized label-propagation problem over an aggregated multilayer graph.
We generate a controlled synthetic instance using an SBM multilayer generator with $N=70$ nodes and $C=5$ communities, with $K_{\mathrm{true}}=3$ informative layers and additional noisy layers so that $K_{\mathrm{true}}/K=0.1$.
Figure~\ref{fig:fw-gaps} (left) reports the evolution of the FW gap for FW/ASFW/PWFW, averaged over $5$ random feasible initializations. 

%\noindent
We next study a data distillation bilevel problem, where the upper-level objective is the validation loss of a linear classifier trained on a weighted subset of training points, and the upper variables are the sample weights constrained by a budget.
We report results on the \textsc{BLOG} dataset \citep{blogfeedback_304} following the setup of \citet{venturini2025relax}.
Figure~\ref{fig:fw-gaps} (right) shows, as for the previous experiments,  the FW gap for FW/ASFW/PWFW.

%\noindent
By taking a look at the plots, we can observe that the FW variants consistently outperform the vanilla FW method in both experimental settings. Interestingly, ASFW and PWFW exhibit a slower decrease of the gap during the initial iterations. This behavior can be attributed to the presence of drop and swap steps, which may yield limited improvement. However, these steps play a crucial role in refining the active set by removing poorly contributing vertices from the current convex combination. As a result, the methods identify a more informative subset of vertices, which subsequently enables a much sharper reduction of the FW gap.

%\noindent
Further problem details and experimental settings for both tasks are provided in Appendix~\ref{sec:CDM_auxiliary}--\ref{sec:DHC_auxiliary}.
The codes are available as a zip file in the supplementary material.

\begin{figure}[H]
  \centering
  \begin{minipage}[t]{0.49\textwidth}
    \centering
    \includegraphics[width=\linewidth]{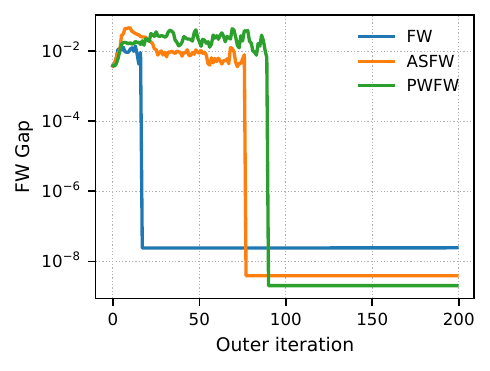}
    %\subcaption{Synthetic dataset (FW gap vs outer iterations, log scale).}
    \label{fig:sbm-strongcomm-gap}
  \end{minipage}\hfill
  \begin{minipage}[t]{0.49\textwidth}
    \centering
    \includegraphics[width=\linewidth]{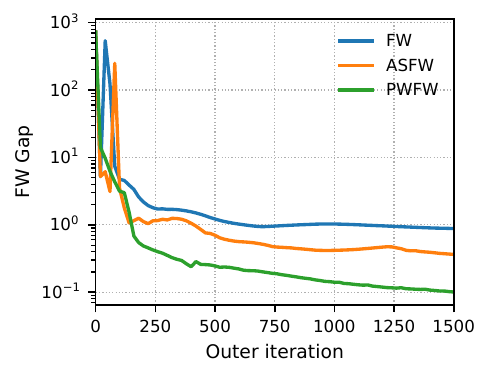}
    %\subcaption{\emph{BLOG} dataset (FW gap vs outer iterations, log scale).}
    \label{fig:blog-gap}
  \end{minipage}

\caption{Frank--Wolfe gap (log scale) versus outer iterations for FW, Away-Step FW (ASFW), and Pairwise FW (PWFW): (a) multilayer semi-supervised learning on synthetic datasets; (b) data distillation on the real \emph{BLOG} dataset.}
  \label{fig:fw-gaps}
\end{figure}

\section{Conclusions and future directions}
We studied projection-free methods for constrained bilevel optimization, where lower-level problems are solved only approximately.
Our analysis shows that both vanilla FW and its away-step and pairwise variants are guaranteed to converge in this setting.

%\noindent
By combining an inexact convergence analysis for the FW variants with existing hypergradient error bounds from iterative and approximate implicit differentiation, we derived complexity  guarantees in terms of lowel-lever iterations of the order 
$
\mathcal{O}\!\left(\tau^{-2}\log(\tau^{-1})\right).
$
To the best of our knowledge, this is the first time that such guarantees are obtained for projection-free methods in the bilevel setting. 

%\noindent
Beyond the theoretical results, our work suggests that FW-type methods provide a sound alternative to projection-based approaches. Future directions include extension of the analysis to stochastic bilevel settings, and application of the method to real-world machine learning problem.

\bibliographystyle{plainnat}
\bibliography{biblio}

@inproceedings{arbel2021,
  title={Amortized implicit differentiation for stochastic bilevel optimization},
  author={Arbel, M and Mairal, J},
  booktitle={International Conference on Learning Representation (ICLR)},
  year={2021}
}

@article{ghadimi2018,
  title={Approximation methods for bilevel programming},
  author={Ghadimi, S and Wang, M},
  journal={arXiv:180202246},
  volume={},
  number={},
  pages={},
  year={2018}
}

@inproceedings{Ji2021,
  title={Bilevel optimization: Convergence analysis and enhanced design},
  author={Ji, K and Yang, J and Liang, Y},
  booktitle={International conference on machine learning},
  pages={4882--4892},
  year={2021},
  organization={PMLR}
}

@inproceedings{venturini2023learning,
  title={Learning the right layers a data-driven layer-aggregation strategy for semi-supervised learning on multilayer graphs},
  author={Venturini, Sara and Cristofari, Andrea and Rinaldi, Francesco and Tudisco, Francesco},
  booktitle={International Conference on Machine Learning},
  pages={35006--35023},
  year={2023},
  organization={PMLR}
}

@inproceedings{grazzi2020iteration,
  title={On the iteration complexity of hypergradient computation},
  author={Grazzi, Riccardo and Franceschi, Luca and Pontil, Massimiliano and Salzo, Saverio},
  booktitle={International Conference on Machine Learning},
  pages={3748--3758},
  year={2020},
  organization={PMLR}
}

@article{bomze2021frank,
  title={Frank--{W}olfe and friends: a journey into projection-free first-order optimization methods},
  author={Bomze, Immanuel M and Rinaldi, Francesco and Zeffiro, Damiano},
  journal={4OR},
  volume={19},
  number={3},
  pages={313--345},
  year={2021},
  publisher={Springer}
}

@book{cgbook,
author = {Braun, Gábor and Carderera, Alejandro and Combettes, Cyrille W. and Hassani, Hamed and Karbasi, Amin and Mokhtari, Aryan and Pokutta, Sebastian},
title = {Conditional Gradient Methods},
publisher = {Society for Industrial and Applied Mathematics},
year = {2025},
doi = {10.1137/1.9781611978568},
address = {Philadelphia, PA},
edition   = {},
URL = {https://epubs.siam.org/doi/abs/10.1137/1.9781611978568},
eprint = {https://epubs.siam.org/doi/pdf/10.1137/1.9781611978568}
}

@inproceedings{franceschi2018bilevel,
  title={Bilevel programming for hyperparameter optimization and meta-learning},
  author={Franceschi, Luca and Frasconi, Paolo and Salzo, Saverio and Grazzi, Riccardo and Pontil, Massimiliano},
  booktitle={International conference on machine learning},
  pages={1568--1577},
  year={2018},
  organization={PMLR}
}

@inproceedings{franceschi2017,
  title={Forward and reverse gradient-based hyperparameter optimization},
  author={Franceschi, Luca and Frasconi, Paolo and Salzo, Saverio and Grazzi, Riccardo and Pontil, Massimiliano},
  booktitle={International conference on machine learning},
  pages={1165--1173},
  year={2017},
  organization={PMLR}
}

@inproceedings{lorraine2020optimizing,
  title={Optimizing millions of hyperparameters by implicit differentiation},
  author={Lorraine, Jonathan and Vicol, Paul and Duvenaud, David},
  booktitle={International conference on artificial intelligence and statistics},
  pages={1540--1552},
  year={2020},
  organization={PMLR}
}

@inproceedings{maclaurin2015,
  title={Gradient-based hyperparameter optimization through reversible learning},
  author={Maclaurin, D and Duvenaud, D and Adams, R},
  booktitle={International conference on machine learning},
  pages={2113--2122},
  year={2015},
  organization={PMLR}
}

@inproceedings{pedregosa2016hyperparameter,
  title={Hyperparameter optimization with approximate gradient},
  author={Pedregosa, Fabian},
  booktitle={International conference on machine learning},
  pages={737--746},
  year={2016},
  organization={PMLR}
}

@article{grazzijmlr23,
  title={Bilevel optimization with a lower-level contraction: optimal sample complexity without warm-start},
  author={Grazzi, Riccardo and Pontil, Massimiliano and Salzo Saverio},
  journal={Journal of Machine Learning Research},
  volume={24},
  pages={1-37},
  year={2023}
}

@article{combettes2021complexity,
  title={Complexity of linear minimization and projection on some sets},
  author={Combettes, Cyrille W and Pokutta, Sebastian},
  journal={Operations Research Letters},
  volume={49},
  number={4},
  pages={565--571},
  year={2021},
  publisher={Elsevier}
}

@inproceedings{jaggi2013revisiting,
  title={Revisiting {F}rank-{W}olfe: Projection-free sparse convex optimization},
  author={Jaggi, Martin},
  booktitle={International conference on machine learning},
  pages={427--435},
  year={2013},
  organization={PMLR}
}

@article{guelat1986some,
  title={Some comments on {W}olfe's 'away step'},
  author={Guélat, Jacques and Marcotte, Patrice},
  journal={Mathematical Programming},
  volume={35},
  pages={pages 110--119},
  year={1986}

}

@article{lacoste2015global,
  title={On the global linear convergence of Frank-Wolfe optimization variants},
  author={Lacoste-Julien, Simon and Jaggi, Martin},
  journal={Advances in neural information processing systems},
  volume={28},
  year={2015}
}

@article{bomze2020active,
  title={Active set complexity of the away-step {F}rank--{W}olfe algorithm},
  author={Bomze, Immanuel M and Rinaldi, Francesco and Zeffiro, Damiano},
  journal={SIAM Journal on Optimization},
  volume={30},
  number={3},
  pages={2470--2500},
  year={2020},
  publisher={SIAM}
}

@article{lacoste2016convergence,
  title={On the convergence of {F}rank-{W}olfe on non-convex objectives},
  author={Lacoste-Julien, Simon},
  journal={arXiv preprint arXiv:1607.00345},
  year={2016}
}

@misc{blogfeedback_304,
  author       = {Buza, Krisztian},
  title        = {{BlogFeedback}},
  year         = {2014},
  howpublished = {UCI Machine Learning Repository},
  note         = {{DOI}: https://doi.org/10.24432/C58S3F}
}

@article{venturini2025relax,
  title={Relax and penalize: a new bilevel approach to mixed-binary hyperparameter optimization},
  author={Venturini, S and De Santis, M and Patracone, J and Schmidt, M and Rinaldi, F and Salzo, S},
  journal={Transactions on Machine Learning Research},
  volume={2025},
  pages={1--27},
  year={2025}
}

%%%%%%%%%%%%%%%%%%%%%%%%%%%%%%%%%%%%%%%%%%%%%%%%%%%%%%%%%%%%%%%%%%%%%%%%%%%%%%%
%%%%%%%%%%%%%%%%%%%%%%%%%%%%%%%%%%%%%%%%%%%%%%%%%%%%%%%%%%%%%%%%%%%%%%%%%%%%%%%
% APPENDIX
%%%%%%%%%%%%%%%%%%%%%%%%%%%%%%%%%%%%%%%%%%%%%%%%%%%%%%%%%%%%%%%%%%%%%%%%%%%%%%%
%%%%%%%%%%%%%%%%%%%%%%%%%%%%%%%%%%%%%%%%%%%%%%%%%%%%%%%%%%%%%%%%%%%%%%%%%%%%%%%

\appendix
\onecolumn
\section{Proofs of the Theoretical Results.}\label{Aproofs}

\begin{proof}[\textbf{Proof of \cref{nonconvb}}]
Let $N\in \N$ be such that, for every $n=0,\dots, N$, $\tilde g_{n}>\tau$.
Hence, recalling Assumption~\ref{assumption},
\begin{equation}\label{cond_inexact_old}
(\forall\,n\in \{0,\dots, N\})\quad 
|(\nabla f(x_n) - \widetilde \nabla f(x_n))^\top (x-x_n)| \leq \frac{\sigma}{1+\sigma} \tilde g_n\quad \forall\  x \in \set \,.
\end{equation}
Therefore, condition (9) in ~\citet[Theorem 4.2]{venturini2023learning} is satisfied, and hence\footnote{Note that our definition of $g_N^*$ corresponds to $g_{N+1}^*$ in \citet[Theorem 4.2]{venturini2023learning}.} 
\begin{equation} \label{g_T_old}
    g_{N}^* \leq \max\left\{\sqrt{\frac{\Delta^2 \Lip (f(x_0) - f^*)}{(N+1) \rho (1 - \sigma)^2}}, \frac{2(f(x_0) - f^*)}{(N+1)(1 - 3\sigma)} \right\}.
\end{equation}
Moreover, we have that, for every $n=0,\dots, N$,
\begin{equation} \label{chain}
    -\nabla f(x_n)^\top d_n^{FW} \geq
    -\nabla f(x_n)^\top d_n \geq
    -\widetilde \nabla f(x_n)^\top d_n - \frac{\sigma}{1+\sigma}\tau \,,
\end{equation}
and hence
\begin{equation} 
\label{eq:20260112c}
\min_{0\leq n \leq N} \tilde g_n \leq \min_{0\leq n\leq N} g_n^{FW} + \frac{\sigma}{1+\sigma}\tau  = g_N^* + \frac{\sigma}{1+\sigma}\tau.
\end{equation}
Therefore, we proved that if, for every $n=0,\dots, N$, $\tilde g_{n}>\tau$, then \eqref{g_T_old}
and \eqref{eq:20260112c} hold. Now, we derive from the bound in \eqref{g_T_old} that if
\begin{equation}
\label{eq:20260112d}
\max\left\{\sqrt{\frac{\Delta^2 \Lip (f(x_0) - f^*)}{(N+1) \rho (1 - \sigma)^2}}, \frac{2(f(x_0) - f^*)}{(N+1)(1 - 3\sigma)} \right\} \leq \frac{\tau}{1+\sigma},
\end{equation}
then $g_N^*\leq \tau/(1+\sigma)$ and hence, recalling \eqref{eq:20260112c},
\begin{equation} 
\min_{0 \leq i \leq N} \tilde g_i \leq g_N^* + \frac{\sigma}{1+\sigma} \leq \frac{\tau}{1+\sigma} + \frac{\sigma}{1+\sigma} = \tau \,,
\end{equation}
which implies that the stopping condition is achieved within the first $N+1$ iterations $\{0,\dots, N\}$,
contradicting the assumptions that, for all $n=0,\dots, N$, $\tilde g_{n}>\tau$. Moreover, condition \eqref{eq:20260112d} is equivalent to 
\begin{equation*}
N+1 \geq N_\tau := \left\lceil\max\left\{
    \frac{\Delta^2 \Lip (f(x_0) - f^*) (1+\sigma)^2}{ \tau^2\rho (1 - \sigma)^2},
    \frac{2 (f(x_0) - f^*) (1+\sigma)}{ \tau (1 - 3\sigma)}
    \right\}
    \right\rceil.
\end{equation*}
In the end, we showed that
\begin{equation*}
(\forall\,n\in \{0,\dots, N\})\quad \tilde{g}^h_{n}>\tau\ \Rightarrow\ N< N_B-1
\end{equation*}
and hence, if $N\geq N_\tau-1$, necessarily there must exists $n\in \{0,\dots, N\}$ such that $\tilde{g}_n^h\leq \tau$. Moreover, for such $N$, \eqref{eq:20260112d} holds and $g_{N}^* \leq \tau/(1+\sigma)$.
\end{proof}

\begin{proof}[\textbf{Proof of  \cref{nonconvb-away}}]
Let $n\in \mathbb{N}$ and define
\begin{equation}
\begin{aligned}
\label{eq:gaps2}
g_n &= -\nabla f(x_n)^\top d_n,\quad \tilde g_n = -\widetilde\nabla f(x_n)^\top d_n,\\ 
g_n^{h} &= -\nabla f(x_n)^\top h_n,\quad
\tilde g_n^{h} = -\widetilde\nabla f(x_n)^\top h_n,
\quad\text{and}\quad g_n^{FW} = -\nabla f(x_n)^\top d_n^{FW},
\end{aligned}
\end{equation}
where $d_n^{FW} \in \argmin_{x \in \set}\{\nabla f(x_n)^\top (x-x_n)\} - x_n$.
The following chain of inequalities holds
\begin{equation} \label{chain}
\begin{aligned}
    -\nabla f(x_n)^\top d_n^{FW} &\geq
    -\nabla f(x_n)^\top h_n \geq
    -\widetilde \nabla f(x_n)^\top h_n - \frac{\sigma}{1+\sigma}\tau\\ 
    &\geq
    -\widetilde \nabla f(x_n)^\top d_n^{FW} - \frac{\sigma}{1+\sigma}\tau \geq
    -\nabla f(x_n)^\top d_n^{FW} - \frac{2\sigma}{1+\sigma}\tau,
    \end{aligned}
\end{equation}
where we used \eqref{cond_inexact} in the second and the last inequality, while the first and the third inequality follow from the definition of $d_n^{FW}$ and $h_n$.
In particular, using the definitions of $\tilde g_n$, $g_n$ and $g_n^{FW}$, from~\eqref{chain} we derive
\begin{align}
g_n^{FW} & \ge \tilde g^h_n - \frac{\sigma}{1+\sigma}\tau, \label{ineq_proof_1} \\
\tilde g^h_n & \ge g_n^{FW} - \frac{\sigma}{1+\sigma}\tau, \label{ineq_proof_2} \\
g^h_n & \ge \tilde g^h_n - \frac{\sigma}{1+\sigma}\tau. \label{ineq_proof_3}
 \end{align}
 Suppose that $N\in \mathbb{N}$ is such that, for every $n= 0,\dots, N$,
$\tilde g^h_{n}>\tau$. Hence, recalling Assumption~\ref{assumption}, and multiplying~\eqref{ineq_proof_1} by $\sigma$, we have that
\begin{equation}\label{chain_2}
(\forall\,n\in \{0,\dots, N\})\quad
\frac{\sigma}{1+\sigma}\tau \leq \sigma \left(\tilde g^h_n - \frac{\sigma}{1+\sigma}\tau\right) \leq \sigma g_n^{FW}.
\end{equation}
Now, for every $n\in \{0,\dots, N\}$, let us distinguish three cases.
\begin{itemize}
\item [\textbf{\textit{A})}]
$\bar{\eta}_n < \eta_n^{\max}$. It follows from~\eqref{alphabound2} that $\displaystyle{\frac{\tilde{g}_n}{{\Lip\n{d_n}^2}}<\eta_n^{\max}}$ and $\bar{\eta}_n = \displaystyle{\frac{\tilde{g}_n}{\Lip\n{d_n}^2}}$. Using~\eqref{eq:rho2} we can write
	\begin{equation*}
	f(x_n) - f(x_n + \eta_n d_n) \ge \rho\bar{\eta}_n\tilde{g}_n = \frac{\rho}{\Lip{\n{d_n}}^2} \tilde{g}_n^2 \ge \frac{\rho (\tilde{g}^h_n)^2} {\Lip\Delta^2},
	\end{equation*}
    where the last inequality follows from $\n{d_n} \leq \Delta$ and the fact that $\tilde g_n\geq \tilde g_n^h$. Observe that, from~\eqref{ineq_proof_2} and~\eqref{chain_2}, we have
    $\tilde g^h_n \ge (1-\sigma) g_n^{FW}$.
    Therefore,
        \begin{equation} \label{eq:caseA}
	f(x_n) - f(x_{n+1}) \ge 
        \frac{\rho (1 - \sigma)^2} {\Delta^2\Lip} (g_n^{FW})^2.
	\end{equation} 
\item [\textbf{\textit{B})}]
$\bar{\eta}_n = \eta_n^{\max}$ and $-\widetilde\nabla f(x_n)^\top h_n \geq -\widetilde\nabla f(x_n)^\top b_n$. Then $d_n = h_n$ and  $\eta_n =\bar{\eta}_n = \eta_n^{\max} = 1$. Moreover, by \eqref{alphabound2} $\tilde{g}_n \geq \Lip\norm{d_n}^2$.
By the standard descent lemma and $\tilde{g}^h_n=\tilde{g}_n$, we can write
\begin{equation*}
 f(x_{n+1}) = f(x_n + d_n) \leq f(x_n) - g^h_n + \frac{\Lip}{2}\|d_n\|^2\,   
\leq f(x_n) - \left(\tilde{g}^h_n - \frac{\sigma}{1+\sigma}\tau\right) + \frac{\Lip}{2}\|d_n\|^2\, ,
\end{equation*}
where we used \eqref{ineq_proof_3} in the last inequality.
Also, since $\tilde{g}^h_n=\tilde g_n \ge \n{d_n}^2\Lip$, we obtain
\begin{equation*}
f(x_n) - f(x_{n+1}) \ge \frac{\tilde{g}^h_n}{2} - \frac{\sigma}{1+\sigma}\tau.
\end{equation*}
Using~\eqref{ineq_proof_2},~\eqref{chain_2}, we also have
\begin{equation*}
\frac{\tilde{g}^h_n}{2} - \frac{\sigma}{1+\sigma}\tau \ge\frac{g_n^{FW}}{2} - \frac{3}{2}\frac{\sigma}{1+\sigma}\tau \ge 
 \frac{g_n^{FW}}{2} - \frac{3}{2}\sigma g_n^{FW} =
 \frac{1 - 3 \sigma}{2} g_n^{FW}.
\end{equation*}
Therefore,
\begin{equation} \label{eq:caseB}
f(x_n) - f(x_{n+1}) \ge
\frac{1 - 3 \sigma}{2} g_n^{FW}.
\end{equation}
\item [\textbf{\textit{C})}]
$\bar{\eta}_n = \eta_n^{\max}$ and $-\widetilde\nabla f(x_n)^\top h_n < -\widetilde\nabla f(x_n)^\top b_n$. Then $d_n = b_n = x_n - \ve_{\hat\j} $ for $\hat\j \in \supp_n$ and $\eta_n= \bar{\eta}_n = \eta_n^{\max}= w_n^{\hat\j}/(1-w_n^{\hat\j})$.
Therefore, for every $j\in \supp_n\setminus\{\hat\j\}$, $w_{n+1}^j = (1+\eta_n)w_n^j>0$
and for $j=\hat\j$, $w_{n+1}^j =(1+\eta_n)w_{n}^{\hat\j}-\eta_n=0$. Thus, $\supp_{n+1} = \supp_n\setminus\{\hat\j\}$ and in particular, $|\supp_{n+1}| = |\supp_n|-1$.
\end{itemize}
%Let $n\in \{1,\dots, N\}$. 
Based on the three cases analyzed above, we partition the iterations $\{0,\ldots,N\}$ into three subsets $\N_A$, $\N_B$, and $\N_C$ defined as follows
\begin{align*}
\N_A &= \big\{n \leq N \colon \bar{\eta}_n < \eta_n^{\max}\big\}\\[0.8ex]
\N_B &= \big\{n \leq N \colon \bar{\eta}_n = \eta_n^{\max}\ \text{and}\ -\widetilde\nabla f(x_n)^\top h_n \geq -\widetilde\nabla f(x_n)^\top b_n\big\},\\[0.8ex]
\N_C &= \big\{n \leq N \colon  \bar{\eta}_n = \eta_n^{\max}\ \text{and}\ -\widetilde\nabla f(x_n)^\top h_n < -\widetilde\nabla f(x_n)^\top b_n\big\}.
\end{align*}
We derive from Remark~\ref{rem:supportS}\ref{rem:support\ve_ii} that
\begin{equation*}
(\forall\, n\in \{ 0,\dots N\})\quad \abs{\supp_{n+1}} \leq 
\begin{cases}
\abs{\supp_n}+1 &\text{if}\ n\in \N_A\cup \N_B\\
\abs{\supp_n}-1 &\text{if}\ n\in \N_C.
\end{cases}
\end{equation*}
Therefore, 
\begin{equation}
\label{Sdiff}
\abs{\supp_{N+1}} \leq 1+  \abs{\N_A}+\abs{\N_B}-\abs{\N_C}.
\end{equation}
Since $\abs{\N_C} = N+1-\abs{\N_A}-\abs{\N_B}$, from \eqref{Sdiff} we get
$\abs{\supp_{N+1}} \leq 1+ 2(\abs{\N_A}+\abs{\N_B})-N-1$ and hence
\begin{equation}
\label{eq:T}
\abs{\N_A}+\abs{\N_B} \geq \frac{N+1+\abs{\supp_N}-1}{2} \geq \frac{N+1}{2}.
\end{equation}
Considering just the good cases \textbf{\textit{A}}) and \textbf{\textit{B}}), where we can have a bound on the decrease, using~\eqref{eq:caseA}, \eqref{eq:caseB}, and the fact that $f(x_{n+1})\leq f(x_n)$ (which is a consequence of \eqref{eq:rho2}, being $\tilde g_n\geq 0$) we have, 
\begin{equation*}
\begin{aligned}
f(x_0)-f^*  &\ge \sum_{n=0}^{N} (f(x_n)-f(x_{n+1})) \\
& \ge \sum_{n\in \N_A} (f(x_n)-f(x_{n+1}))+
\sum_{n\in \N_B} (f(x_n)-f(x_{n+1})) \\
& \ge \sum_{k\in \N_A} \frac{\rho (1 - \sigma)^2} {\Delta^2\Lip}(g_k^{FW})^2 +
\sum_{n\in \N_B} \frac{1 - 3 \sigma}{2} g_n^{FW} \\
& \ge \abs{\N_A} \min_{n \in \N_A} \frac{\rho (1 - \sigma)^2} {\Delta^2\Lip}(g_n^{FW})^2 + \abs{\N_B} \min_{n \in \N_B} \frac{1 - 3 \sigma}{2} g_n^{FW} \\
& \ge (\abs{\N_A}+\abs{\N_B}) \min \left\{\frac{\rho (1 - \sigma)^2} {\Delta^2 \Lip}(g_{N}^{*})^2 ,\frac{1 - 3 \sigma}{2} g_N^{*} \right\} \\
& \ge \frac{N+1}{2} \min \left\{\frac{\rho (1 - \sigma)^2} {\Delta^2\Lip}(g_{N}^{*})^2 ,\frac{1 - 3 \sigma}{2} g_N^{*} \right\},
\end{aligned}
\end{equation*}
where in the second-last inequality we use the definition of $g_N^{*}$ and the last inequality is due to~\eqref{eq:T}.\\
Hence,
\begin{align*}
\frac{\rho (1 - \sigma)^2} {\Delta^2\Lip}(g_N^{*})^2 \le \frac{1 - 3 \sigma}{2} g_N^{*} \; & \Rightarrow  \;
g^*_{N} \leq \sqrt{\frac{2 \Delta^2 \Lip (f(x_0) - f^*)}{(N+1) \rho ( 1 - \sigma)^2}}, \\
\frac{\rho (1 - \sigma)^2} {\Delta^2\Lip}(g_N^{*})^2 > \frac{1 - 3 \sigma}{2} g_N^{*} \; & \Rightarrow  \;
g^*_{N} \le \frac{4(f(x_0) - f^*)}{(N+1)(1 - 3\sigma)} .
\end{align*} 
It follows that, 
\begin{equation}\label{g_T_old2}
g_{N}^* \leq \max\left\{\sqrt{\frac{2\Delta^2 \Lip (f(x_0) - f^*)}{(N+1) \rho (1 - \sigma)^2}}, \frac{4(f(x_0) - f^*)}{(N+1)(1 - 3\sigma)} \right\} \,.
\end{equation}
Moreover, from inequality \eqref{ineq_proof_1}, we have that,
\begin{equation} 
\label{eq:20260112a}
%(\forall\,k=0,\dots, n-1)\quad \tilde g^h_k \leq g_k^{FW} + \frac{\sigma}{1+\sigma}\tau \quad \Rightarrow\quad 
\min_{0\leq n \leq N} \tilde g^h_n \leq \min_{0\leq n\leq N}g_n^{FW} + \frac{\sigma}{1+\sigma}\tau = g_N^* + \frac{\sigma}{1+\sigma}\tau \,.
\end{equation}
In the end we proved that if for every $n=0,\dots, N$, $\tilde{g}^h_{n}>\tau$, then 
\eqref{g_T_old2} and \eqref{eq:20260112a} hold.
Now, it follow from the bound in  \eqref{g_T_old2} that if $N$ is such that
\begin{equation}
\label{eq:20260112b}
\max\left\{\sqrt{\frac{2\Delta^2 \Lip (f(x_0) - f^*)}{(N+1) \rho (1 - \sigma)^2}}, \frac{4(f(x_0) - f^*)}{(N+1)(1 - 3\sigma)} \right\} \leq  \frac{\tau}{1+\sigma},
\end{equation}
then $g_N^* \leq \tau/(1+\sigma)$ and hence
\begin{equation} 
\min_{0 \leq n \leq N} \tilde{g}^{h}_n \leq g_N^* + \frac{\sigma}{1+\sigma} \leq \frac{\tau}{1+\sigma} + \frac{\sigma}{1+\sigma} = \tau \,,
\end{equation}
implying that the stopping condition is achieved within the first $N+1$ iterates $\{0,\dots, N\}$ and contradicting the assumption that for all $n=0,\dots, N$, $\tilde{g}^h_{n}>\tau$. Moreover, condition in \eqref{eq:20260112b} is equivalent to 
\begin{equation*}
N +1\geq N_\tau:=\left\lceil 
    \max\left\{
    \frac{2 \Delta^2 \Lip (f(x_0) - f^*) (1+\sigma)^2}{ \tau^2\rho (1 - \sigma)^2},
    \frac{4 (f(x_0) - f^*) (1+\sigma)}{ \tau (1 - 3\sigma)}
    \right\}
    \right\rceil. 
\end{equation*}
This shows that 
\begin{equation*}
(\forall\,n\in \{0,\dots, N\})\quad \tilde{g}^h_{n}>\tau\ \Rightarrow\ N< N_\tau-1
\end{equation*}
and hence, if $N\geq N_\tau-1$, necessarily there must exists $n\in \{0,\dots, N\}$ such that $\tilde{g}_n^h\leq \tau$. Moreover, for such $N$, \eqref{eq:20260112b} holds and hence $g_{N}^* \leq \tau/(1+\sigma)$.
\end{proof}

\begin{proof}[\textbf{Proof of  \cref{nonconvbp-pairwise}}]
We define $g_n, \tilde{g}_n, g_n^h, \tilde{g}_n^h$, and $g^{FW}_n$ as in \eqref{eq:gaps2},
so that inequalities \eqref{ineq_proof_1}--\eqref{ineq_proof_3} hold.
 Suppose that $N\in \mathbb{N}$ is such that, for every $n= 0,\dots, N$,
$\tilde g^h_{n}>\tau$. Hence, recalling Assumption~\ref{assumption}, and multiplying~\eqref{ineq_proof_1} by $\sigma$, we have that
\begin{equation}\label{chain_2p}
(\forall\,n\in \{0,\dots, N\})\quad
\frac{\sigma}{1+\sigma}\tau \leq \sigma \left(\tilde g^h_n - \frac{\sigma}{1+\sigma}\tau\right) \leq \sigma g_n^{FW}.
\end{equation}
Now, for every $n\in \{0,\dots, N\}$, let us distinguish two cases.

\begin{itemize}
\item [\textbf{\textit{A})}]
$\bar{\eta}_n < \eta_n^{\max}$ (PW, AS, FW). It follows from~\eqref{alphabound2p} that $\displaystyle{\frac{\tilde{g}_n}{{\Lip\n{d_n}^2}}<\eta_n^{\max}}$ and $\bar{\eta}_n = \displaystyle{\frac{\tilde{g}_n}{\Lip\n{d_n}^2}}$. Using~\eqref{eq:rho2p} we can write
\begin{equation}
\label{eq:20260122b}
f(x_n) - f(x_n + \eta_n d_n) \ge \rho\bar{\eta}_n\tilde{g}_n = \frac{\rho}{\Lip{\n{d_n}}^2} \tilde{g}_n^2 \ge \frac{\rho (\tilde{g}_n)^2} {\Lip\Delta^2},
\end{equation}
where the last inequality follows from $\n{d_n} \leq \Delta$.
Now, we consider two subcases. If $x_{n+1} \vcentcolon= x_n + \eta_n d_n$ with $d_n = h_n+ b_n$ (a pairwise step is performed), by definition of $b_n$,  we have
\begin{align*}
-\widetilde \nabla f(x_n)^\top b_n &= \widetilde \nabla f(x_n)^\top (\ve_{\hat \j} - x_n)\\[0.8ex]
&= \max_{j\in \ve_n}\widetilde \nabla f(x_n)^\top (\ve_{j} - x_n)\\ 
&= \max_{x\in \conv(\{\ve_j\,\vert\, j\in \supp_n\}))} \widetilde \nabla f(x_n)^\top (x - x_n)\geq 0.
\end{align*}
Thus, 
\begin{equation}
\label{eq:20260120a} 
\tilde g_n= \tilde g_n^h-\widetilde \nabla f(x_n)^\top b_n\geq \tilde g_n^h.
\end{equation}
Moreover, from~\eqref{ineq_proof_2} and~\eqref{chain_2}, we have
$\tilde g^h_n \ge (1-\sigma) g_n^{FW}$.
Therefore, \eqref{eq:20260122b} yields
\begin{equation} \label{eq:caseAp}
f(x_n) - f(x_{n+1}) \ge 
\frac{\rho (1 - \sigma)^2} {\Delta^2\Lip} (g_n^{FW})^2.
\end{equation}	
On the other hand, if $x_{n+1} \vcentcolon= x_n + \eta_n d_n$ with $d_n$ determined by the rule in lines \ref{line:activeset_h1_start}-\ref{line:activeset_h1_end} of Algorithm~\ref{alg:SAFW}, then we can follow the argument in case \textbf{\textit{A})} of the proof of Theorem~\ref{nonconvb} and conclude that
\eqref{eq:caseAp} still holds.
\item [\textbf{\textit{B})}] $\bar{\eta}_n = \eta_n^{\max}=1$ (PW FW). Clearly $\eta_n=1$. Suppose that 
$x_{n+1} \vcentcolon= x_n + \eta_n d_n$ with $d_n = h_n+ b_n$ (PW). Then $\eta_n^{\max} = 1= w_n^{\hat\j}\ \Rightarrow\ \ve_n= \{\hat\j\}\ \Rightarrow\ x_n = \ve_{\hat\j}\ \Rightarrow\ b_n=0\ \Rightarrow\ d_n = h_n$. On the other hand we have again $x_{n+1} \vcentcolon= x_n + \eta_n d_n$ with $d_n = h_n$ even when $-\widetilde\nabla f(x_n)^\top h_n \geq -\widetilde\nabla f(x_n)^\top b_n$ and lines \ref{line:activeset_h1_start}-\ref{line:activeset_h1_end} of Algorithm~\ref{alg:SAFW} are executed. Thus, in any case we are in the same situation considered in the case \textbf{\textit{B})} of the proof of Theorem~\ref{nonconvb-away}. Therefore we have the sufficient decrease    
\begin{equation} \label{eq:caseBp}
f(x_n) - f(x_{n+1}) \ge
\frac{1 - 3 \sigma}{2} g_n^{FW}.
\end{equation}
\item [\textbf{\textit{C})}]
$\bar{\eta}_n = \eta_n^{\max}\neq 1$  (PW AS). Clearly $\eta_n=\eta_n^{\max}\neq 1$ and we have 
$x_{n+1} \vcentcolon= x_n + \eta_n d_n$ with either $d_n = h_n+ b_n$ or
 $d_n = b_n$ (since line \ref{line:20260122a} of Algorithm~\ref{alg:SAFW} implies $\eta_n^{\max}=1$).
Therefore, by Remark~\ref{rem:supportS}\ref{rem:support\ve_ii} and Remark~\ref{rmk:pairwiseFW} have two subcases:
\begin{itemize}
\item [\textbf{\textit{$C_1$})}] If $\hat\i \in \ve_n$, then $\ve_{n+1} = \ve_n\setminus \{\hat\j\}$, a \emph{drop step} is executed and $|\ve_{n+1}| = |\ve_n|-1$; 
\item [\textbf{\textit{$C_2$})}]otherwise we have a \emph{swap step}  $|\ve_{n+1}| = |\ve_n|$. 
\end{itemize}
%So, $|\ve_n|\geq |\ve_{n+1}| \geq |\ve_n|-1$.
 \end{itemize}
Based on the three cases analyzed above, we partition the iterations $\{0,\ldots,N\}$ into three subsets $\N_A$, $\N_{B_1}$, and $\N_{B_2}$ defined as follows
\begin{align*}
\mathbb{N}_A &= \big\{n \leq N \colon \bar{\eta}_n < \eta_n^{\max}\big\},\\[0.8ex]
\N_{B} &= \big\{n \leq N \colon \eta_n=1 (PW, FW)\big\},\\[0.8ex]
\N_{C_1} &= \big\{n \leq N \colon  \bar{\eta}_n = \eta_n^{\max},\, \eta_n\neq 1 (PW, AS), \hat\i\in \supp_n\big\}\\[0.8ex]
\N_{C_2} &= \big\{n \leq N \colon  \bar{\eta}_n = \eta_n^{\max},\, \eta_n\neq 1 (PW)\big\}.
\end{align*}
We derive from Remark~\ref{rem:supportS}\ref{rem:support\ve_ii} and Remark~\ref{rmk:pairwiseFW} that
\begin{equation*}
(\forall\, n\in \{ 0,\dots N\})\quad \abs{\supp_{n+1}} \leq 
\begin{cases}
\abs{\supp_n}+1 &\text{if}\ n\in \N_A\cup \N_{B} \\
\abs{\supp_n}&\text{if}\ n\in \N_{C_2}\\
\abs{\supp_n}-1 &\text{if}\ n\in \N_{C_1}.
\end{cases}
\end{equation*}
Therefore, 
\begin{equation}
\label{Sdiffa}
\abs{\supp_{N+1}} \leq 1+  \abs{\N_A}+\abs{\N_{B}} -\abs{\N_{C_1}}.
\end{equation}
Since $\abs{\N_{C_1}} = N+1-\abs{\N_A}-\abs{\N_{B}}- \abs{\N_{C_2}}$, from \eqref{Sdiff} we get
$\abs{\supp_{N+1}} \leq 1+ 2(\abs{\N_A}+\abs{\N_B})+\abs{\N_{C_2}}-N-1$ and hence
\begin{equation}
\label{eq:T}
\abs{\N_A}+\abs{\N_B} \geq \frac{N+1+\abs{\supp_{N+1}}-1- \abs{\N_{C_2}}}{2} \geq \frac{N+1- \abs{\N_{C_2}}}{2}.
\end{equation}
Now we note that
\begin{equation*}
\abs{\N_{C_2}}\leq m R\ \ \text{and}\ \ m(R+1)\leq N+1\ \text{for some}\ m\in \N.
\end{equation*}
Therefore,
\begin{equation*}
\abs{\N_{C_2}} \leq \frac{R}{R+1} (N+1)
\end{equation*}
In the end, recalling \eqref{eq:T} we have
\begin{equation*}
\abs{\N_A}+\abs{\N_B} \geq \frac{N+1}{2} \bigg(1-\frac{R}{R+1} \bigg) = \frac{N+1}{2(R+1)}
\end{equation*}
Continuing as in the proof of Theorem~\ref{nonconvb-away} the statement follows.
Considering just the good cases \textbf{\textit{A}}) and \textbf{\textit{B}}), where we can have a bound on the decrease, using~\eqref{eq:caseAp}, \eqref{eq:caseBp}, and the fact that $f(x_{n+1})\leq f(x_n)$ we have, 
\begin{equation*}
\begin{aligned}
f(x_0)-f^*  &\ge \sum_{n=0}^{N} (f(x_n)-f(x_{n+1})) \\
& \ge \sum_{n\in \N_A} (f(x_n)-f(x_{n+1}))+
\sum_{n\in \N_B} (f(x_n)-f(x_{n+1})) \\
& \ge \sum_{k\in \N_A} \frac{\rho (1 - \sigma)^2} {\Delta^2\Lip}(g_k^{FW})^2 +
\sum_{n\in \N_B} \frac{1 - 3 \sigma}{2} g_n^{FW} \\
& \ge \abs{\N_A} \min_{n \in \N_A} \frac{\rho (1 - \sigma)^2} {\Delta^2\Lip}(g_n^{FW})^2 + \abs{\N_B} \min_{n \in \N_B} \frac{1 - 3 \sigma}{2} g_n^{FW} \\
& \ge (\abs{\N_A}+\abs{\N_B}) \min \left\{\frac{\rho (1 - \sigma)^2} {\Delta^2\Lip}(g_{N}^{*})^2 ,\frac{1 - 3 \sigma}{2} g_N^{*} \right\} \\
& \ge \frac{N+1}{2(R+1)} \min \left\{\frac{\rho (1 - \sigma)^2} {\Delta^2\Lip}(g_{N}^{*})^2 ,\frac{1 - 3 \sigma}{2} g_N^{*} \right\},
\end{aligned}
\end{equation*}
where in the second-last inequality we use the definition of $g_N^{*}$ and the last inequality is due to~\eqref{eq:T}.\\
Hence,
\begin{align*}
\frac{\rho (1 - \sigma)^2} {\Delta^2\Lip}(g_N^{*})^2 \le \frac{1 - 3 \sigma}{2} g_N^{*} \; & \Rightarrow  \;
g^*_{N} \leq \sqrt{\frac{2(R+1) \Delta^2 \Lip (f(x_0) - f^*)}{(N+1) \rho ( 1 - \sigma)^2}}, \\
\frac{\rho (1 - \sigma)^2} {\Delta^2\Lip}(g_N^{*})^2 > \frac{1 - 3 \sigma}{2} g_N^{*} \; & \Rightarrow  \;
g^*_{N} \le \frac{4(R+1)(f(x_0) - f^*)}{(N+1)(1 - 3\sigma)} .
\end{align*} 
It follows that, 
\begin{equation}\label{g_T_old3}
g_{N}^* \leq \max\left\{\sqrt{\frac{2(R+1)\Delta^2 \Lip (f(x_0) - f^*)}{(N+1) \rho (1 - \sigma)^2}}, \frac{4(R+1)(f(x_0) - f^*)}{(N+1)(1 - 3\sigma)} \right\} \,.
\end{equation}
Moreover, from inequality \eqref{ineq_proof_1}, we have that,
\begin{equation} 
\label{eq:20260128a}
%(\forall\,k=0,\dots, n-1)\quad \tilde g^h_k \leq g_k^{FW} + \frac{\sigma}{1+\sigma}B \quad \Rightarrow\quad 
\min_{0\leq n \leq N} \tilde g^h_n \leq \min_{0\leq n\leq N}g_n^{FW} + \frac{\sigma}{1+\sigma} = g_N^* + \frac{\sigma}{1+\sigma}B \,.
\end{equation}
In the end we proved that if for every $n=0,\dots, N$, $\tilde{g}^h_{n}>B$, then 
\eqref{g_T_old3} and \eqref{eq:20260128a} hold.
Now, it follow from the bound in \eqref{g_T_old2} that if $N$ is such that
\begin{equation}
\label{eq:20260128b}
\max\left\{\sqrt{\frac{2(R+1)\Delta^2 \Lip (f(x_0) - f^*)}{(N+1) \rho (1 - \sigma)^2}}, \frac{4(R+1)(f(x_0) - f^*)}{(N+1)(1 - 3\sigma)} \right\} \leq  \frac{B}{1+\sigma},
\end{equation}
then $g_N^* \leq B/(1+\sigma)$ and hence
\begin{equation} 
\min_{0 \leq n \leq N} \tilde{g}^{h}_n \leq g_N^* + \frac{\sigma}{1+\sigma} \leq \frac{B}{1+\sigma} + \frac{\sigma}{1+\sigma} = B \,,
\end{equation}
implying that the stopping condition is achieved within the first $N+1$ iterates $\{0,\dots, N\}$ and contradicting the assumption that for all $n=0,\dots, N$, $\tilde{g}^h_{n}>B$. Moreover, condition in \eqref{eq:20260128b} is equivalent to 
\begin{equation*}
N +1\geq N_B:=\left\lceil 
    \max\left\{
    \frac{2 (R+1) \Delta^2 \Lip (f(x_0) - f^*) (1+\sigma)^2}{ B^2\rho (1 - \sigma)^2},
    \frac{4(R+1) (f(x_0) - f^*) (1+\sigma)}{ B (1 - 3\sigma)}
    \right\}
    \right\rceil. 
\end{equation*}
This shows that 
\begin{equation*}
(\forall\,n\in \{0,\dots, N\})\quad \tilde{g}^h_{n}>B\ \Rightarrow\ N< N_B-1
\end{equation*}
and hence, if $N\geq N_B-1$, necessarily there must exists $n\in \{0,\dots, N\}$ such that $\tilde{g}_n^h\leq B$. Moreover, for such $N$, \eqref{eq:20260112b} holds and hence $g_{N}^* \leq B/(1+\sigma)$.
\end{proof}

%%%%%%%%%%%%%%%%%%%%%%%%%%%%%%%%%%%%%%%%%%%%%%%%%%%%%%%%%%%%%%%%%%%%%%%%%%%%%%%
%%%%%%%%%%%%%%%%%%%%%%%%%%%%%%%%%%%%%%%%%%%%%%%%%%%%%%%%%%%%%%%%%%%%%%%%%%%%%%%

% We now prove that the exact stepsize 
% \[
% \bar\eta_n = \min\!\left\{\eta_n^{max}, \frac{\tilde g_n}{\Lip \|d_n\|^2}\right\}
% \; 
% \]
% satisfies  conditions \eqref{alphabound}-\eqref{eq:rho} needed to guarantee convergence of the methods. 
% \par\medskip
% \begin{lemma}
% The stepsize rule $\eta_n = \bar\eta_n$ satisfies  conditions \eqref{alphabound}-\eqref{eq:rho}.
% \end{lemma}

\begin{proof}[Proof of Lemma \ref{lemma:es}]
  Reasoning as in the proof of Theorem~\ref{nonconvb}, inequality~\eqref{chain} holds, namely
\begin{equation} \label{ineq:g}
 g_n \geq  \tilde g_n - \frac{\sigma}{1+\sigma}\tau \,,
\end{equation}
where \(g_n = -\nabla f(x_n)^\top d_n\) and \(\tilde g_n = -\widetilde{\nabla} f(x_n)^\top d_n\).
By the descent lemma and the Lipschitz continuity of \(\nabla f\) with constant \(\Lip\), for any \(\eta\geq 0\) we have
\begin{equation}\label{eq:descent_appendix}
f(x_n) - f(x_n + \eta d_n)
\;\ge\;
\eta g_n - \frac{\Lip}{2}\eta^2 \|d_n\|^2.
\end{equation}
Now, we can write:
\begin{equation}\label{eq:descent_appendix2}
\eta g_n - \frac{\Lip}{2}\eta^2 \|d_n\|^2
\;\ge\;
\eta (\tilde g_n - \frac{\sigma}{1+\sigma}\tau ) - \frac{\Lip}{2}\eta^2 \|d_n\|^2
\;\ge\;
\eta \tilde g_n \left(1 - \frac{\sigma}{1+\sigma}\right ) - \frac{\Lip}{2}\eta^2 \|d_n\|^2
,
\end{equation}
where we used equation \eqref{ineq:g} in the first inequality and 
$\tilde g_n>\tau$ in the second one. 
It is immediate to check that 
\[
\eta \tilde g_n \left(1 - \frac{\sigma}{1+\sigma}\right ) - \frac{\Lip}{2}\eta^2 \|d_n\|^2
\;\ge\ \eta \tilde g_n\left(1 - \frac{\sigma}{1+\sigma}\right ) -\frac{\eta \tilde g_n}{2}
=
\frac{1-\sigma}{2(1+\sigma)}\eta \tilde g_n.
\]
for every $0\leq \eta \leq \frac{\tilde g_n}{L\|d_n\|^2}.$ Hence, considering that $\sigma$ satisfies Assumption \ref{assumption}, then there exists \(\rho=\frac{1-\sigma}{2(1+\sigma)}>0\) such that
\[
f(x_n) - f(x_n + \eta_n d_n)
\;\ge\;
\rho\, \bar\eta_n \tilde g_n,
\]
so condition~\eqref{eq:rho} holds. 
It is easy to see that equation \eqref{alphabound} is satisfied when $\eta_n=\bar \eta_n$ and the result is proved.
\end{proof}

\section{Experiments Details}\label{Appendix:Exp}
\subsection{Community Detection on Multilayer Networks}
\label{sec:CDM_auxiliary}

We follow the data-driven layer aggregation model of \citet{venturini2023learning}.
Let $G^{(k)}=(V,E^{(k)},W^{(k)})$ be the $k$-th layer over a common node set $V$ with $|V|=N$, and let $K$ be the total number of available layers.
Let $Y_{\mathrm{tr}}\in\mathbb{R}^{N\times C}$ denote the training labels in one-hot form (rows corresponding to unlabeled nodes are zero), and let $(V_{\mathrm{val}},y_{\mathrm{val}})$ denote the validation set, where $V_{\mathrm{val}}\subseteq V$ and $y_{\mathrm{val}}:V_{\mathrm{val}}\to\{1,\dots,C\}$.

%\noindent
For $\alpha\in\mathbb{R}$ and $\beta\in\Delta_K:=\{\beta\in\mathbb{R}^K_+:\sum_{k=1}^K\beta_k=1\}$, we aggregate the $K$ layer weights into a single weighted graph via the generalized mean
\begin{equation}
\label{eq:gmean_app}
W(\alpha,\beta)_{ij}
=
\left(\sum_{k=1}^K \beta_k \bigl(W^{(k)}_{ij}\bigr)^{\alpha}\right)^{1/\alpha},
\end{equation}
with the limit $\alpha\to 0$ handled as a weighted geometric mean.
This defines an aggregated adjacency matrix and hence a Laplacian $L(\alpha,\beta)$.
In all experiments we constrain the hyperparameters to
\[
\alpha\in[-2,2],\qquad
\beta\in\Delta_K,\qquad
\lambda\in[0.01,1],
\]
where $\lambda$ is the Laplacian regularization parameter.

%\noindent
Given $(\alpha,\beta,\lambda)$, the inner variable $X\in\mathbb{R}^{N\times C}$ is obtained by solving the Laplacian-regularized SSL problem (quadratic label-propagation case)
\begin{equation}
\label{eq:lower_level_app}
X(\alpha,\beta,\lambda)\in\argmin_{X\in\mathbb{R}^{N\times C}}
\;\; \|X-Y_{\mathrm{tr}}\|_F^2
\;+\;\frac{\lambda}{2}\,\mathrm{Tr} \ \!\bigl(X^\top L(\alpha,\beta)\,X\bigr).
\end{equation}
The upper-level objective is the multiclass cross-entropy evaluated on the validation labels using the inner solution $X(\alpha,\beta,\lambda)$:
\begin{equation}
\label{eq:upper_level_app}
\min_{\alpha,\beta,\lambda}\;
\mathrm{CE} \ \!\bigl(X(\alpha,\beta,\lambda);V_{\mathrm{val}},y_{\mathrm{val}}\bigr)
:=
-\frac{1}{|V_{\mathrm{val}}|}
\sum_{i\in V_{\mathrm{val}}}
\log\!\left(\frac{\exp(X_{i,y_{\mathrm{val}}(i)})}{\sum_{c=1}^C \exp(X_{i,c})}\right),
\end{equation}
subject to $(\alpha,\beta,\lambda)\in[-2,2]\times\Delta_K\times[0.01,1]$.
Equations \eqref{eq:lower_level_app}--\eqref{eq:upper_level_app} define the bilevel SSL task.

%\noindent
We generate $K_{\mathrm{true}}=3$ \emph{true} layers from an SBM with $N=70$ nodes and $C=5$ communities.
A base community assignment is sampled once, and for each layer $k$ we sample edges independently with probability $p_{\mathrm{in}}^{(k)}$ within communities and $p_{\mathrm{out}}^{(k)}$ across communities; nonzero edge weights are sampled uniformly in $[0.5, 1.5]$.
To enforce partial coherence across layers, we apply a \emph{layer drift} mechanism: between consecutive layers, a fixed fraction of nodes is reassigned to a (uniformly chosen) different community.

%\noindent
To make layer selection nontrivial and emphasize sparse solutions (where FW variants are expected to be advantageous), we append additional \emph{noisy} layers to the $K_{\mathrm{true}}$ true layers so that $K_{\mathrm{true}}/K=0.1$.
Noisy layers are generated as random graphs with density matched to the average density of the true layers and with independent (non-community) structure; they are treated identically by the aggregation model.
Table~\ref{tab:sbm_regimes} summarizes the parameters used for the \emph{true} layers.
\begin{table}[t]
  \caption{SBM regimes used for the \emph{true} layers ($K_{\mathrm{true}}=3$) with $N=70$ and $C=5$.
  Layer drift is $0.10$.
  Noisy layers are appended so that $K_{\mathrm{true}}/K=0.1$.}
  \label{tab:sbm_regimes}
  \begin{center}
    \begin{small}
      \begin{sc}
        \begin{tabular}{cccr}
          \toprule
           $N$ & $C$ & $(p_{\mathrm{in}}^{(k)})_{k=1}^3$ & $(p_{\mathrm{out}}^{(k)})_{k=1}^3$ \\
          \midrule
          $70$ & $5$ & $(0.35,\,0.30,\,0.25)$ & $(0.03,\,0.04,\,0.05)$ \\
          % $70$ & $5$ & $(0.22,\,0.20,\,0.18)$ & $(0.10,\,0.11,\,0.12)$ \\
          \bottomrule
        \end{tabular}
      \end{sc}
    \end{small}
  \end{center}
  \vskip -0.1in
\end{table}

%\noindent
Labeled nodes are split into $80\%$ train and $20\%$ test/validation.
Only $10\%$ of the training labels are revealed (all remaining nodes are treated as unlabeled).
To reduce sensitivity to initialization, we use a multistart protocol: for each algorithm we perform $5$ independent runs, each initialized from a randomly sampled feasible point $\theta_0=(\alpha_0,\beta_0,\lambda_0)$ with $\alpha_0\sim\mathrm{Unif}([-2,2])$, $\lambda_0\sim\mathrm{Unif}([0.01,1])$, and $\beta_0$ sampled on the simplex.

%\noindent
For each method we run $200$ outer iterations, and each hypergradient evaluation uses at most $500$ iterations of the inner solver.
The lower-level objective in \eqref{eq:lower_level_app} has Lipschitz-continuous gradient with constant
$L_{\mathrm{in}}=\|2I+\lambda L(\alpha,\beta)\|_2$.
We set the inner stepsize to $\eta_{\mathrm{in}}=1/\widehat L_{\mathrm{in}}$, where $\widehat L_{\mathrm{in}}$ is estimated by power iteration.
At the upper level, we estimate a Lipschitz constant $M$ of the hypergradient by sampling $10$ feasible points and taking the maximum pairwise ratio of hypergradient differences to parameter differences.
For the Frank--Wolfe methods (FW, ASFW, PWFW) we use the stepsize rule of \eqref{alphabound}, and \eqref{alphabound2}, \eqref{alphabound2p}, respectively.

\subsection{Data Distillation}
\label{sec:DHC_auxiliary}
We consider data distillation, where the goal is to identify a compact, high-fidelity summary of a training set by assigning nonnegative weights to training samples under a budget constraint.
Given a training dataset $\mathcal{D}^{\mathrm{train}}=\{(x_i^{\mathrm{train}},y_i^{\mathrm{train}})\}_{i=1}^{m}$ and a validation dataset $\mathcal{D}^{\mathrm{val}}=\{(x_j^{\mathrm{val}},y_j^{\mathrm{val}})\}_{j=1}^{n}$, we introduce sample weights $v\in[0,1]^m$ constrained by a distillation budget $B\in\mathbb{Z}_+$ via
\[
\mathcal{Y}:=\{v\in[0,1]^m:\; \mathbf{1}^\top v = B\}.
\]
The upper-level objective evaluates validation performance of a model trained on the weighted training set.

%\noindent
Let $\zeta$ denote the model parameters. We solve
\begin{equation}
\label{eq:distill_bilevel}
\min_{v\in\mathcal{Y}}
\;\; \ell^{\mathrm{val}}(\zeta(v))
\qquad \text{s.t.}\qquad
\zeta(v)\in\argmin_{\zeta}\;
\ell^{\mathrm{train}}(\zeta,v) + \frac{s}{2}\|\zeta\|^2,
\end{equation}
where $s>0$ is a regularization parameter and
$\ell^{\mathrm{train}}(\zeta,v)=\sum_{i=1}^m v_i\,\ell(x_i^{\mathrm{train}},y_i^{\mathrm{train}};\zeta)$
is the weighted training loss. After solving the continuous problem, a discrete distilled set can be obtained by selecting the top-$B$ samples with the largest weights.

%\noindent
We consider a linear classifier $\zeta=(W,b)$ with $W\in\mathbb{R}^{e\times d}$ and $b\in\mathbb{R}^e$.
Both $\ell^{\mathrm{train}}$ and $\ell^{\mathrm{val}}$ are cross-entropy (CE) losses:
\[
L^{\mathrm{train}}(W,b,v)=\frac{1}{m}\sum_{i=1}^m v_i\,\mathrm{CE}(W x_i^{\mathrm{train}} + b,\,y_i^{\mathrm{train}}),
\qquad
L^{\mathrm{val}}(W,b)=\frac{1}{n}\sum_{j=1}^n \mathrm{CE}(W x_j^{\mathrm{val}} + b,\,y_j^{\mathrm{val}}).
\]
The feasible set for the upper variables is the capped simplex (budget) polytope $\mathcal{Y}$.
For the FW variants (FW/ASFW/PWFW), linear minimization over $\mathcal{Y}$ admits an efficient oracle.

%\noindent
We report results on the \textsc{BLOG} dataset \citep{blogfeedback_304}, which consists of feature vectors extracted from blog posts, with the associated response variable being the number of comments received within the next 24 hours. For a detailed description, see~\citet{venturini2025relax}.

%\noindent
We use a distillation budget $B=0.1\%$.
The lower-level problem is approximately solved by gradient descent for $50$ iterations with stepsize $10^{-7}$ and regularization parameter $s=2\cdot10^3$.
At the upper level, we estimate the Lipschitz constant $M$ of the hypergradient by sampling $10$ feasible points and taking the maximum pairwise ratio of hypergradient differences to parameter differences.
For the Frank-Wolfe methods (FW, ASFW, PWFW), we use the theoretical upper bounds on the step size given in \eqref{alphabound}, and \eqref{alphabound2}, \eqref{alphabound2p}, respectively. 
We initialize the upper variable by distributing the budget uniformly, $v_i^{(0)}=B/m$.
The lower-level parameters are initialized with $W_{ab}\sim\mathcal{N}(0,0.01^2)$ and $b^{(0)}=0$.

\end{document}